\documentclass[12pt]{amsart}  
\usepackage{amssymb,color,comment}
\usepackage{enumerate}
  \usepackage{xspace,xcolor}

  \usepackage{tikz}
\usetikzlibrary{decorations.markings}
\usetikzlibrary{arrows}
\usetikzlibrary{calc}
  \usepackage{tkz-euclide}

\usepackage[
     breaklinks,colorlinks,
     citecolor=blue,linkcolor=blue,urlcolor=teal
  ]{hyperref}

\newtheorem{theorem}{Theorem}[section]
\newtheorem{corollary}[theorem]{Corollary}
\newtheorem{addendum}[theorem]{Addendum}

\newtheorem{lemma}[theorem]{Lemma}

\newtheorem{proposition}[theorem]{Proposition}
\newtheorem{prop}[theorem]{Proposition}

\theoremstyle{definition}
\newtheorem{remark}[theorem]{Remark}
\newtheorem{example}[theorem]{Example}
\newtheorem{definition}[theorem]{Definition}

\def\G{\Gamma}

\def\R{\mathbb{R}}
\def\Z{\mathbb{Z}}
\def\Q{\mathbb{Q}}

\def\FF{\mathcal{AF}}
\def\AF{\mathcal{AF}}
\def\FFn{\mathcal{AF}_n}
\def\OF{\mathcal{OF}}

\def\GL{{\rm{GL}}}
\def\PGL{{\rm{PGL}}}
\def\aut{{\rm{Aut}}}
\def\Aut{{\rm{Aut}}}
\def\Out{{\rm{Out}}}

\def\autn{{\rm{Aut}}(F_n)}

\def\<{\langle}
\def\>{\rangle}

\def\D{\Delta}

\def\L{\Lambda}

\def\core{{\rm{core}}}
\def\injrad{{\rm{injrad }}}

\def\fold{{\rm{fold}}}
\def\sub{{\rm{sub}}}
\def\rk{{\rm{rank}}}
\def\rank{{\rm{rank}}}
\def\Lk{{\rm{Lk}}}
\def\E{{\rm{E}}}

\def\A{\mathcal{A}}
\def\O{\mathcal{O}}

\def\tits{{\rm{Tits}}}

\def\e{\varepsilon}

\title{Rigidity of the free factor complex}
\author[Bestvina and Bridson]{Mladen Bestvina and Martin R. Bridson}
\address{Department of Mathematics\\
  University of Utah\\
  Salt Lake City, UT 84103, USA}
\address{Mathematical Institute\\
Andrew Wiles Building\\ ROQ\\ Oxford\\ OX26GG\\ U.K.}

\begin{document}
\begin{abstract}{We establish the following non-abelian analogue of the Fundamental Theorem of Projective Geometry:
the natural map from $\autn$ to the automorphism
group of the free-factor complex $\FF_n$  is an isomorphism. We also prove the corresponding 
theorem for the action of $\Out(F_n)$ on the complex of conjugacy classes of free factors.} 
\end{abstract}

\maketitle

\section{Introduction}   

Our purpose in this article is to describe the symmetries of the complex of free factors $\FF_n$ associated to a 
finitely generated free group $F_n$. We shall prove that the natural map from $\autn$ to the automorphism
group of $\FF_n$  is an isomorphism.  We shall also prove the corresponding 
theorem for the action of $\Out(F_n)$ on the complex of conjugacy classes of free factors. These results
can be viewed as  non-abelian analogues of the Fundamental Theorem of Projective Geometry, as we shall now explain. 

The Fundamental Theorem of Projective Geometry \cite{vonS} 
establishes that, for any field $K$,  
the only bijections of a projective space over $K$  
that preserve incidence relations are the natural ones, i.e. combinations of field automorphisms
and projective-linear maps. This can be rephrased 
in terms of the {\em Tits building} $\tits_n^<(K)$, 
which is the poset of proper non-trivial subspaces of $K^n$. If $K=\Q$ then there are no field automorphisms 
and the theorem tells us that the natural map $\PGL(n,\Q)\to \aut(\tits_n^<(\Q))$ is an isomorphism provided $n\ge 3$.  
The geometric realisation $\tits_n(\Q)$ of this poset has an additional symmetry: 
its group of simplicial automorphisms  
is $\PGL(n,\Q)\rtimes\Z/2$,  with the generator of $\Z/2$ swapping each vertex $V$ with $V^\perp$,
where the orthogonal complement is taken with respect to an inner product on $\Q^n$. (This is an anti-isomorphism of
the poset $\tits_n^<(\Q)$.)

The inclusion $\Z^n\hookrightarrow\Q^n$ induces an isomorphism $\mathcal{D}_n(\Z)\to\tits_n^<(\Q)$, where 
$\mathcal{D}_n(\Z)$ is the poset of proper direct factors of $\Z^n$, ordered by inclusion. Passing from the free abelian group $\Z^n$ to the non-abelian free group $F_n$, the natural analogue of $\mathcal{D}_n(\Z)$ is the poset of non-trivial proper {\em free factors} of $F_n$, ordered by inclusion. We shall work with the geometric realisation of this poset, which we denote by $\FFn$.
This complex was introduced by Allen Hatcher and Karen Vogtmann \cite{HV, HV2} who used it to study the
cohomology of $\autn$; they proved, in analogy with the Solomon-Tits theorem for $\tits_n(\Q)$, that $\FFn$
has the homotopy type of a wedge of spheres of dimension $n-2$.

As in the classical case, one has to assume $n\ge 3$ in order to obtain the desired rigidity for the automorphism group of 
this complex.

\begin{theorem}\label{thm1}
For $n\geq 3$ the natural homomorphism $\aut(F_n)\to \aut(\FF_n)$ is an isomorphism. 
\end{theorem}

Note in particular that every automorphism of $\FF_n$ preserves the type of each
vertex, i.e. the rank of each  free factor;  there is no equivalent of the involution $V\leftrightarrow V^\perp$ 
of $\tits_n(\Q)$. 

We also prove a version of the above theorem for $\Out(F_n)$. In this case, the natural analogue of $\tits_n(\Q)$ is
the geometric realisation  $\OF_n$ of the poset of conjugacy classes of
non-trivial proper free factors in $F_n$, i.e. the quotient $\FF_n/{\rm{Inn}}(F_n)$. 
The large-scale geometry of $\OF_n$ was elucidated by Bestvina and Feighn \cite{BF}, who proved that it is 
a space of infinite diameter that is hyperbolic in the
sense of Gromov. 

\begin{theorem}\label{thm2}
For $n\geq 3$ the natural homomorphism $\Out(F_n)\to \Aut(\OF_n)$ is an
isomorphism. 
\end{theorem}

A key similarity between $\tits_n(\Q)$, on the one hand, and $\FF_n$ and
$\OF_n$ on the other, is that each is composed of {\em{standard apartments}}. In the case of
$\tits_n(\Q)$, such an apartment is the full subcomplex whose vertices represent
the subspaces spanned by the proper, non-empty subsets of a basis for $\Q^n$.
A standard apartment in $\FF_n$ is defined in much the same way, taking the free factors
 spanned by the non-empty proper subsets of a basis.  
In each case, an apartment is simplicially isomorphic to the
barycentric subdivision of the boundary of an $(n-1)$-simplex. 

There are also important differences between $\tits_n(\Q)$ and $\FF_n$. The
former is a spherical building of
diameter 3, while $\FF_n$ has infinite diameter. From a technical
point of view, a major difficulty in understanding the automorphisms of $\FF_n$ comes from the fact that, in 
contrast to $\tits_n(\Q)$,
there are many ``fake apartments'' in $\FF_n$, i.e. subcomplexes abstractly
isomorphic to the 
barycentric subdivision of the boundary of an $(n-1)$-simplex that are
not standard apartments (Section \ref{s:fakes}).

The first stage in our proof of  Theorem \ref{thm1} involves establishing another difference, to which we have already alluded:
every simplicial automorphism of $\FF_n$ preserves the partial ordering on the vertex set, i.e. the rank of  free
factors; this is achieved in Section \ref{s:rank}.

Our aim in the  second stage of the proof  (Section \ref{s:standard}) 
is to show that standard apartments can be recognized intrinsically: they can be distinguished from fake apartments by metric properties of their neighbourhoods. From this it follows that the set of standard apartments
is preserved by all automorphisms of $\FF_n$.
The key technical result in this part of the proof is the {\em Antipode Lemma} (Theorem \ref{p:antipode}), which 
provides an intrinsic (metric) characterisation of pairs of vertices $A,L$ such that $A\ast L =F_n$. 

In the third stage of the proof, working outwards from a fixed standard apartment,
we consider adjacent apartments that have large overlaps. 
A key role is played in this part of the argument by {\em sticks} --  
certain rank 1 factors that, when gathered in appropriate families,  
provide rigid, highly-symmetric frames controlling 
large overlaps between apartments (see Section \ref{s:sticks}).  

With these tools in hand, the final step in our proof is straightforward: $\aut(F_n)$ 
acts transitively on the set of standard apartments,
preserving the rank of vertices,
so by composing an arbitrary automorphism $\Phi$ of $\FF_n$ with a suitable element of $\aut(F_n)$, we may assume that
$\Phi$ fixes a standard apartment; we argue that one can compose with a further element of $\aut(F_n)$ to ensure
that $\Phi$ fixes the apartment and all of the adjacent sticks pointwise; this forces $\Phi$ to fix the neighbouring apartments
and their sticks pointwise (Proposition \ref{p:fix-sticks}), and by propagation $\Phi$ is forced to be the identity everywhere.

Our proof of Theorem \ref{thm2}  follows the same outline but there are some additional difficulties to address, 
notably that it is harder to recognise standard apartments, which are no longer uniquely determined by their
rank 1 vertices.

The parallel that we focussed on to motivate Theorem \ref{thm1} compared
$\FF_n$  to  $|{\mathcal{D}}_n(\Z)| \cong \tits_n(\Q)$. This
is a facet of the powerful 3-way analogy between automorphism groups of free groups, lattices such as
${\rm{SL}}(n,\Z)$, and mapping class groups of surfaces of finite type \cite{mb-icm, BV-survey}. In this grand analogy,
the object corresponding to
$\FFn$ and $\OF_n$  in the setting of mapping class groups is the curve complex \cite{harvey}.  Ivanov \cite{ivan} proved the analogue of Theorems \ref{thm1} and \ref{thm2} in this setting: the natural map from the extended mapping class group of a surface of finite type  to the group of simplicial
automorphisms of the corresponding curve complex is an isomorphism (with some exceptions for small surfaces -- cf.~\cite{kork}, \cite{luo}).

Ivanov used his theorem to deduce that the extended mapping class group of a 
surface of finite type is equal to its own abstract commensurator (with the same exceptions for small surfaces).
In connection with this,  
we should comment on the fact that $\Aut(\FF_n)$ is $\Aut(F_n)$,
whereas $\Aut(\mathcal{D}_n(\Z))$ is ${\PGL}(n,\Q)$ not  ${\PGL}(n,\Z)$. This difference can be interpreted as a 
manifestation of the fact that ${\GL}(n,\Q)$ is the abstract commensurator of ${\GL}(n,\Z)$.  
In contrast, commensurations of $\Aut(F_n)$ (i.e.~isomorphisms between subgroups of finite index)
are as restricted as they are in the mapping class group case: 
Bridson and Wade \cite{BW} prove that the action of $\Aut(F_n)$ on $\FF_n$ extends to a faithful action  by 
 ${\rm{Comm}}(\Aut(F_n))$,  
and it then follows from Theorem \ref{thm1}  that $\Aut(F_n)={\rm{Comm}}(\Aut(F_n))$. 
The corresponding result for $\Out(F_n)$ is due to Farb and Handel \cite{FarbH} for $n\ge 4$ 
and to Horbez and Wade \cite{HW2} for $n\ge 3$ (with proofs that do not follow the template we have described).

Theorems \ref{thm1} and \ref{thm2} also extend the range of faithful geometric models  
for $\Aut(F_n)$ and $\Out(F_n)$ --  by which we mean spaces $X$  where a
natural action induces an isomorphism $\Aut(F_n)\to {\rm{Aut}}(X)$ or $\Out(F_n)\to {\rm{Aut}}(X)$.
The first such rigidity result was proved by Bridson and Vogtmann, who showed that $\Out(F_n)$ is the group
of simplicial automorphisms of the spine of Outer space \cite{BV-duke}. Other such spaces $X$ include  
the simplicial closure of Outer space \cite{AS}, the free and cyclic splitting complexes \cite{AS, HW1}, 
and Outer space endowed with the Lipschitz metric \cite{franco}. This last result, due to 
Francaviglia and Martino, is the natural analogue
of Royden's theorem on the isometries of Teichm\"{u}ller space \cite{royden}, which was reproved by
Ivanov \cite{ivan} using the rigidity of the curve complex (the analogue of Theorem \ref{thm1}), with an argument 
modelled on the proof of Mostow rigidity in higher rank \cite{mostow}, which in turn relies on understanding  
the automorphisms of spherical buildings such as $\tits_n(\R)$, which
is where we  began.

{\bf Acknowledgements.} The first author gratefully acknowledges the
support by the National Science Foundation under grant number
DMS-1905720.

\section{Background and Preliminaries}

We shall assume that the reader is familiar with basic algebraic facts about free groups and
their subgroups. For example, if $L<F_n$ is a free factor and $H<F_n$ then $H\cap L$ is a free factor of $H$; in particular any intersection of free factors in $F_n$ is a free factor. 

Throughout this paper we shall explore subgroups of free groups by working with labeled graphs that represent them. 
In this section we gather a range of facts that we shall need concerning these graphical representations.

\subsection{Labeled graphs and Stallings folds}

We fix a basis $\{a_1,\dots,a_n\}$ for $F_n$ and identify $F_n$ with the fundamental group of the rose $R_n$,
which is a graph\footnote{we allow graphs to have multiple edges and loops}
 with one vertex $v$ and $n$ edges, directed and labeled $a_1,\dots,a_n$. The length of a word 
 $w$ in the letters $a_i^{\pm 1}$  (equivalently, an edge path in
 $R_n$) will be denoted by $|w|$. A {\it morphism} of
 graphs is a continuous map that sends vertices to vertices and
 edges to edges.
Formally, a {\em labeled graph} is a morphism of 
graphs $\lambda:\G\to R_n$;  
in practice, we  regard $\G$ as a graph in which the edges have been oriented and labeled by letters $a_i$
so that $\lambda$ preserves the orientation and labeling. Given $H<F_n$, the {\em pointed} $\core_*(H)$
is the labeled graph obtained by restricting the (based) covering map $(\widetilde{R}_n,\ast)/H\to (R_n,v)$ to the minimal connected subgraph containing all the embedded loops and the basepoint, while the
(unpointed)  $\core(H)\subset \core_*(H)$ is the
minimal connected subgraph containing all the embedded loops. $H_1$ is conjugate to $H_2$ if and only if
$\core(H_1)=\core(H_2)$.

If a pair of directed edges $e, e'$ in a labeled graph $\G$ have the
same label and the same initial (resp. terminal) vertex, then the
morphism of labeled graphs $\G\to \G'$ that identifies these edges
and their terminal (resp. initial) vertices is called a Stallings {\em
  fold}, \cite{stallings}. Any morphism of finite graphs can be
expressed as a finite sequence of folds followed by an immersion
(locally injective map).  There is a unique graph $\fold(\G)$ obtained from $\G$
by a maximal sequence of folds; such a graph is said to be {\em fully
  folded}; its labeling map $\fold(\G)\to R_n$ is an immersion.

We say that a labeled graph with basepoint $(\G,\ast)$ {\em supports} 
a subgroup $K<F_n$ if $K$ is contained in the $\pi_1$-image of the
labeling map $\G\to R_n$.  

For labeled graphs $\G_1$ and $\G_2$  with basepoints, 
$\G_1\vee\G_2$ will denote the labeled graph obtained from $\G_1\sqcup \G_2$ by identifying the basepoints.
We refer to $\G_1\vee\G_2$  as the {\em wedge} of $\G_1$ and $\G_2$. 
If $\G_1=\core_*(H_1)$ and $\G_2=\core_*(H_2)$, then $\fold(\G_1\vee\G_2)=\core_*\<H_1,H_2\>$. 
The following special case of this observation will be useful.

\begin{lemma} A subgroup
$H<F_n$ of rank $k$ is a free factor if and only if there is a labeled graph $\G$ of rank $(n-k)$
such that $\core_*(H)\vee \G$ folds to $R_n$.  
\end{lemma}

The following well-known lemma is proved by observing how a graph of rank 1 can fold into  $\core(L_{n-1})$.

\begin{lemma}\label{complement} Let $L_{n-1}=\<a_1,\dots,a_{n-1}\>$. 
Then   $L_{n-1}*\<u\>=F_n$ if and only if $u=xa_n^{\pm 1}y$
for some $x,y\in L_{n-1}$. 
\end{lemma}

The following criterion for recognising factors of corank 1 will also be useful.

\begin{prop}\label{p:identify} If $H<F_n$ is a  free factor of rank $n-1$, then either
$\core_*(H)$ embeds in the rose $R_n$ or else the labeled graph obtained by identifying two
of its vertices folds to $R_n$.
\end{prop}

\begin{proof}
Choose $u\in F_n$ such that $H\ast\<u\>=F_n$. We add a loop
labeled $u$ to $\G=\core_*(H)$ at $*$ and start folding to obtain $R_n$. Initially, at every
step an edge of the $u$-loop folds with an edge of $\Gamma$. If the
process stops before the whole loop is folded in,  
$\core_*(H)$ embeds in $R_n$. Otherwise, when the last edge of the $u$-loop is
folded in, two vertices of  $\core_*(H)$ will be identified before the folding to $R_n$ continues.
\end{proof}

 \subsection{Concerning visible factors and powers}\label{s:2.2}
 
The following standard facts will be used without further comment throughout the paper; the second is
used in the proof of the lemma that follows.
 
\begin{enumerate} 
\item[•] If $H_1<H_2$ then there is a unique label-preserving immersion $\core_*(H_1)\to \core_*(H_2)$
restricting to an immersion $\core(H_1)\to \core(H_2)$.
\item[•] $u\in F_n$ is conjugate into $H<F_n$ if and only if there is an oriented loop in $\core(H)$ whose label
is a cyclically reduced word representing the conjugacy class of $u$; if $H$ is malnormal (e.g. a free factor) then there is a unique such loop (up to rotation).  
\end{enumerate}

We need an elaboration on the second point. To explain this, recall that 
the set $L_H$ 
of reduced words representing the elements of a finitely generated subgroup $H<F_n$ 
consists of the labels on the reduced edge paths
in $\core_*(H)$ that begin and end at the basepoint. This sits inside $\sub(H)$, the set of labels on all reduced paths
in $\core_*(H)$,
i.e. words $v$ such that some $uvw$ is a reduced word in $L_H$. Define
$$
\E_{a_i}(H) = \{ n \mid a_i^n \in  \sub(H)\}.
$$

\begin{lemma} If $H$ is a free factor, then
 $\E_{a_i}(H)$ is infinite if and only if $\core_*(H)$ has a loop labeled by the basis element $a_i$.
\end{lemma}

\begin{proof}
Any edge path in $\core_*(H)$ whose length exceeds the number of vertices will contain a loop, and a shortest such loop 
along the path will be embedded.
So if  $\E_{a_i}$ is infinite then $\core_*(H)$ contains 
an embedded loop labeled $a_i^m$ for some $m\neq 0$. This embedded loop represents the
conjugacy class of a primitive element,  so $|m|=1$. The converse is obvious.
\end{proof}

\begin{corollary}\label{c:loops}
Let $A< F_n$ be a free factor of  rank $n-1$. Then, either $\E_{a_i}(A)$ is finite for some $i\ge 2$, or else
$\core_*(A)$ is a tree with $n-1$ loops attached, labeled $a_2,\dots,a_n$. 
\end{corollary}

\begin{proof}
If $\E_{a_i}(A)$ is infinite for each  $i\ge 2$, then the lemma provides loops labeled $a_i$, and since the rank of $\core_*(A)$ is $n-1$, the remainder of the graph is a tree.
\end{proof}
 
\begin{lemma}\label{lift}
Let $V<F_{n}$  
be a free factor of rank $n-1$
and assume that both $\<a_3,\cdots,a_n\>$ and $\<a_2,\cdots,a_{n-1}\>$
  can be conjugated into $V$.
\begin{enumerate}[(1)]
\item If $n\geq 4$ then $V$ is conjugate to $\<a_2,\cdots,a_n\>$.
\item If $n=3$ then $V$ is conjugate to $\<a_2^{\gamma},a_3\>$ for some
  $\gamma\in F_3$.
\end{enumerate}
\end{lemma}

\begin{proof}
In this proof factors are considered up to conjugacy so we ignore
basepoints and work with $\core(V)$.
 
The assumptions imply that the inclusions $\core(\<a_3,\cdots,a_n\>)\hookrightarrow R_n$ and
$\core(\<a_2,\cdots,a_{n-1}\>)\hookrightarrow R_n$ both lift to $\core(V)\to
R_n$. If $n\geq 4$ these lifts both contain the unique loop of $\core(V)$ labeled $a_3$, so they
overlap  and their union is a wedge of $n$ loops labeled $a_2,\cdots,a_n$, thus proving (1).

If $n=3$ we know only that $\core(V)$ contains embedded loops labeled
$a_2$ and $a_3$. As $\core(V)$ has no vertices of valence $1$
and $\rank(V)=2$, it must be the
graph obtained from these two loops by connecting them with an arc,
labeled $\gamma$ say. This proves (2).
\end{proof}

The case $n\ge 4$ in the preceding lemma can also be deduced from the 
following consequence of the second bullet point above.
 
\begin{lemma}\label{fig-8}
If $V<F_n$ is a free factor that contains conjugates of $a_1, a_2$ and $a_1a_2$, then 
the loops labeled $a_1$ and $a_2$ are based at the same vertex of $\core (V)$. 
\end{lemma}

\begin{proof}
The union of the loops in $\core (V)$ labeled $a_1, a_2$ and $a_1a_2$ is equal in homology
to the union of the loops  labeled $a_1$ and $a_2$, because  
$H_1(V)$ injects into $H_1(F)$. It follows that these subgraphs coincide, and hence the loop labelled $ab$
is based at the same vertex as either the $a$-loop or the $b$-loop, forcing all three loops to be based at the same vertex.
\end{proof}

\subsection{Intersections and pullbacks}

Given finitely generated $H_1,H_2<F_n$ one can compute the
intersection $H_1\cap H_2$ by constructing the pullback of the
labeling maps $\core_*(H_i)\to R_n$: the vertex set of the pullback
graph $P$ consists of pairs of vertices $(v,v')\in\core_*(H_1)\times
\core_*(H_2)$ with the same image in $R_n$, and the directed edges of
$P$ are pairs of directed edges with the same image in $R_n$.  The
component of $P$ that contains the basepoint $(*,*)$ is
$\core_*(H_1\cap H_2)$, possibly with trees attached. Some of the
components of $P$ may be trees, while those with non-trivial fundamental
group correspond  to the non-trivial intersections of $H_1$ with
the conjugates of $H_2$.

\subsection{Free factor graphs, distance in $\FF_n$ and $\OF_n$, and links}

$\FF_n$ is the geometric realisation of the poset of non-trivial
proper free factors of $F_n$ ordered by inclusion. 
For $n\ge 3$ it is a flag complex, so every automorphism of its 1-skeleton $\FF_n^{(1)}$ extends uniquely to a
simplicial automorphism of $\FF_n$. Thus, studying the group of simplicial automorphisms of $\FF_n$ is equivalent to 
studying the group of isometries of the graph $\FF_n^{(1)}$, metrized so that each edge has length $1$. 
To lighten the notation, we sometimes write $\FF_n$ in place of $\FF_n^{(1)}$, when concentrating on the
{\em free factor graph}, which has vertices the non-trivial free factors $A<F_n$ and has an edge joining $A$ to $B$
if $A<B$. Similarly, rather than studying $\OF_n$ as a simplicial complex we shall sometimes
concentrate on its 1-skeleton,
i.e.  the quotient of the free factor graph by the action of ${\rm{Inn}}(F_n)$ -- so vertices are conjugacy classes of 
proper free factors and there is an edge from $[A]$ to $[B]$ if there are representatives of these conjugacy classes
with $A<B$.  

When $n\ge 3$, we write $d_\A(A,B)$ for the combinatorial distance between vertices in (the 1-skeleton of) $\FF_n$
 and $d_\O([A],[B])$ for the distance in $\OF_n$. When there is no danger of ambiguity, we will simply write $d$.
 We shall use the terms ``automorphism" and ``isometry" interchangeably and supress mention of
 the restriction isomorphism from the group of simplicial automorphisms of the full complex $\FF_n$ to
 the isometry group of its 1-skeleton, writing both groups as ${\rm{Isom}}(\AF_n)$ or ${\rm{Aut}}(\AF_n)$ (and similarly for $\OF_n$).

We shall not have to bother much with the case $n=2$, but when we
do we must modify the above definition because $\FF_2$ is just a discrete set: to account for this we 
regard $\FF_2$ as the vertex set of the graph that has an edge joining
$\<a\>$ to $\<b\>$ whenever $\<a,b\>=F_2$ and metrize it and $\OF_2$ accordingly. (This makes $\OF_2$
isometric to the vertex set of the Farey graph.)

Estimating distances and understanding neighbourhoods 
in $\FF_n$ and $\OF_n$ is difficult in general, as we shall see  in the  proof of 
Theorems \ref{thm1} and \ref{thm2}, but there are some simple facts relating distance to the algebra of free factors.
For example:

\begin{lemma}\label{l:le2}
Let $V_1$ and $V_2$ be vertices in $\FF_n$ with $\rank(V_1)=n-1$. Then $d(V_1,V_2)\leq 2$ if and only
if $V_1\cap V_2\neq 1$.
\end{lemma}

\begin{proof} $d(V_1,V_2)=1$ if and only if $V_2<V_1$.
If $d(V_1,V_2)=2$, then there is a free factor $U$ with $d(V_1,U)=1=d(U,V_2)$, whence
 $U<V_1$ and either $V_2<U$ or  $U<V_2$. In the first case $V_2<V_1$
and in the second case $U\subset V_1\cap V_2$.
\end{proof}

A similar argument establishes:

\begin{lemma}\label{l:ole2}
Let $[V_1]$ and $[V_2]$ be vertices in $\OF_n$ with $\rank(V_1)=n-1$. Then $d([V_1],[V_2])\leq 2$ if and only
if $V_1^w\cap V_2\neq 1$ for some $w\in F_n$.
\end{lemma} 

There will be certain points in our argument where it is convenient to work with the whole complex $\FF_n$ rather than
just its 1-skeleton. This is particularly true of arguments that involve links $\Lk(V)$. The following observations are
useful in induction arguments.

\begin{lemma}\label{l:links}
If $V\in\FF_n$ has rank $n-1$, then there is a rank-preserving isomorphism
$\Lk(V)\cong \FF_{n-1}$. More generally, if $\rank(V)=k$ then
the subcomplex  $\Lk_-(V)\subset\Lk(V)$ spanned by vertices of rank less than $k$ is isomorphic to $\FF_{k}$.
Similarly, if $[V]\in\OF_n$ has rank $k$,  then
 $\Lk_-[V]\subset\Lk[V]$ is isomorphic to $\OF_{k}$.
\end{lemma}

 \begin{proof}
 The assertions about $\FFn$ are immediate from the definitions. For the assertion about $\OF_n$ 
 one needs to note that because $V$ is malnormal in $F_n$, free factors $A,A'<V$ are conjugate in
 $V$ if they are conjugate in $F_n$.
 \end{proof} 
 
 We shall also need the following observation concerning links.
 
 \begin{lemma}\label{l:lk1} 
If $[C]\in \OF_n$ is a vertex of rank 1, then
$\Lk_{\OF_n}([C])\cong \Lk_{\AF_n}(C)$.
\end{lemma}

\begin{proof} Without loss of generality we may assume $C=\<a_1\>$. If $d([C], [L])=1$,
then $L$ contains a conjugate of $a_1$ and
$\core(L)$ contains a unique loop labeled $a_1$.   
We select a conjugate $L_{a_1}\in [L]$ by decreeing
the vertex at which this loop is based to be the basepoint. This choice 
$[L]\mapsto L_{a_1}$ provides an inverse to the canonical projection $\Lk_{\AF_n}(C)\to \Lk_{\OF_n}([C])$.
\end{proof}

\subsection{Fully irreducible automorphisms and injectivity radius}

Recall that an automorphism $f:F_n\to F_n$ is called {\em fully irreducible} if for every proper free factor $A<F_n$
and every $k>1$, the free factor $f^k(A)$ is {\em not} conjugate to $A$. 
Fully irreducible automorphisms exist in every rank
$n\ge 2$,  \cite{PV}. The results in this section are valid for an {\bf{arbitrary fully irreducible automorphism $f$}}
but when we come to use them in Section \ref{s:antipode} we will be free to fix a choice, so it would be
enough, for example, to prove these results for $a_1\mapsto a_2\mapsto a_1a_2$ in the case $n=2$, or
\begin{equation}
f_0: a_1\mapsto a_2\mapsto a_3\mapsto\dots \mapsto a_{n-1}\mapsto a_n\mapsto a_1 a_3 a_4\dots a_n a_2,
\end{equation}
in the general case. 
\medskip

The following proposition can be proved using standard facts 
about stable laminations \cite{BFH97}. We give an alternative proof suited to the study of free factor
complexes; the general theory is hidden in our appeal to \cite{BF}.   

\begin{proposition} \label{l:loxo} Let $f\in{\rm{Aut}}(F_n)$ be a fully irreducible automorphism.
For all $\ell >0, R>0$ and every free factor $A< F_n$, there is an integer $K=K(f,A,\ell,R)$ such that,
 all $k\ge K$, 
$$
d_\A(f^k(A), C) \ge d_\O([f^k(A)], [C]) \ge R
$$
for all rank-1 free factors $C=\<c\>$ with $|c|\le \ell$.
\end{proposition}

\begin{proof}
The first inequality is obvious. As there are only finitely many rank-1 free factors with $|c|\le \ell$, 
the second inequality is an immediate consequence of the fact (Theorem 9.3 of \cite{BF}) that fully irreducible elements
act on $\OF_n$ as isometries with positive translation length, so   $\inf\{ d_\O(f^k(V), V) \mid V\in \OF_n\} > k\lambda_f$
with $\lambda_f>0$, hence
$$d_\O([f^k(A)], [C]) \ge k\lambda_f - d_\O([A], [C]).$$ 
\end{proof}

In the above proof it was overkill to use the fact that orbits of $f$ grow linearly: we only needed the orbits to be
unbounded.  

\begin{corollary}\label{c:injrad}
$\injrad(\core(f^k(A)))\to\infty$ as $k\to\infty$.
\end{corollary}

\begin{proof} If $c$ is a word  of length $\ell$ labeling an embedded loop in $\core(f^k(A))$, then $C=\<c\>$ is a rank-1
free factor conjugate into  $f^k(A)$, so $d_\O([f^k(A)], [C])=1$, which contradicts the proposition unless $k< K(f,A, \ell, 2)$.
\end{proof}

\begin{corollary}\label{l:no-tiddler}
Let $A<F_n$ be a factor of rank $n-1$ and fix $\ell>0$. 
If $k$ is sufficiently large, then $\< f^k(A), w\> = F_n$ implies that every word conjugate to $w$ has length at
least $\ell$.
\end{corollary}

\begin{proof} Let $C=\<w\>$; it is a free factor. Then $\< f^k(A), w\>= F_n$ implies 
$d(f^k(A), C) = 3$ when $n\ge 3$ or $d(f^k(A), C) = 1$ when $n=2$. The lemma tells us
that this cannot happen if $k$ is sufficiently large and $w$ is
conjugate to a word of length less than $\ell$.
\end{proof}

\subsection{Subfactor Projections}\label{s:subfactor}

Subfactor projections were introduced by Bestvina and Feighn  in \cite{subfactorproj}. The definition and use
of subfactor projections
is motivated by the theory of subsurface projections introduced by Masur and Minsky \cite{MM}. 
For $n\geq 3$, if $A<F_n$ is a free factor of rank $n-1$, then the subfactor projection $\pi_{A}$ 
assigns to suitable vertices $[B]\in \OF_n$ a subcomplex  $\pi_{A}([B])$  of uniformly bounded diameter in
the free factor complex $\Lk[A]\cong\OF_{n-1}$.

In more detail (see \cite{sam}), $\pi_{A}$ is defined on $[B]\neq [A]$ provided  that $[B]$ 
does not contain a conjugate $B^w$ antipodal to $A$ and that it has the following properties:
    \begin{itemize}
    \item the diameter of $\pi_{A}([B])$  is uniformly bounded
    \item if $B$ is conjugate into $A$ then $\pi_{A}([B])=[B]$ 
    \item $\pi_{A}$ is coarsely Lipschitz, i.e.~there is a constant $\delta$ such that
     if $d_\O([B], [C]) = 1$ and  $\pi_{[A]}$ is defined for both $[B]$ and $[C]$, then
     the Hausdorff distance between $\pi_{A}([B])$ and  $\pi_{A}([C])$ is at most $\delta$.
    \end{itemize}     

\section{Automorphisms of $\FF_n$ preserve the rank of vertices}\label{s:rank}  

Let $\FF_n(i)\subset\FF_n$ be the set of vertices of rank $i$. 
The following proposition is the first step in the proof of Theorem \ref{thm1}.

\begin{prop}\label{rank i} Every
automorphism of $\FF_n$ preserves $\FF_n(i)$ for $i=1,\dots,n-1$.
\end{prop}

The proof is broken into several preliminary results. 

\begin{lemma}\label{1 and n-1}
Every automorphism of $\FF_n$ preserves $\FF_n(1)\cup\FF_n(n-1)$. 
\end{lemma}

\begin{proof} 
When $n=3$ there is nothing to prove, so assume that $n>3$. 
In this proof it is convenient to work with the whole complex $\FF_n$
rather than just the 1-skeleton.
If $A$ is
a factor of rank $i$ with $1<i<n-1$ then $\Lk(A)$ can be written as the
join $\Lk_-(A)*\Lk_+(A)$, where $\Lk_-(A)\cong\FF_i$ is the full subcomplex spanned by factors
contained in $A$ and $\Lk_+(A)$ is the full subcomplex spanned by factors  containing
$A$.  
To finish the
proof we need to argue that links of vertices of rank 1 and $n-1$ are
not joins. We will argue that they have diameter greater than $2$.

In the case of a rank $n-1$ factor $A$, we have $\Lk(A)\cong \FF_{n-1}$. As ${\rm{Aut}}(F_n)$ acts transitively on the set of factors of each rank, we may  assume $A=\<a_1,\dots,a_{n-1}\>$.
We could appeal to the non-trivial fact that $ \FF_{n-1}$ has infinite diameter, but it is easy to see
that it has diameter at least $3$, which suffices here: by Lemma \ref{l:le2}, it is
enough to exhibit a rank $1$ free factor $C<A$ and a rank $n-2$ free factor
$B<A$ such that $C\cap B=1$; let $C=\<a_1\>$ and let $B=\<a_2,\dots,a_{n-1}\>$. 

For the rank 1 case we examine the link of $\<a_1\><F_n$, focusing on
$\<a_1,a_2\>$ and $\<a_1,a_3,\dots,a_n\>$.  The intersection of these
factors is $\<a_1\>$ so, arguing as in the proof of Lemma \ref{l:le2},
we see that their distance in the link is greater than $2$.
\end{proof}

To distinguish rank $1$ vertices from rank $n-1$ vertices, we examine the geometry of their
 neighbourhoods in $\FF_n$.

 \begin{lemma}\label{l:link-antip}
 Let $A<F_n$ be a free factor of rank $n-1$, 
let $C=\<u\>$ be a free factor of rank $1$, and suppose $F_n= A\ast C$.
For any vertex $L$, if $d(A,L)=1$ then $d(L,C)=2$.
 \end{lemma}

\begin{proof} If $L<A$ then $C$ is not contained in $L$ and $V=\<L, C\>=L\ast C$ is a free factor with $d(L,V)=d(V,C)=1$.
\end{proof} 

This lemma says that a geodesic from $C$ to $A$ (which has length $3$) cannot be extended to a geodesic of length $4$; indeed any extension will
necessarily backtrack towards the initial vertex $C$. We shall prove
that this metric property fails if we reverse the roles of rank $1$
and rank $n-1$ vertices, that is, we find extensions that don't
backtrack. 

\begin{proposition}\label{p:extend}
For every rank $1$ vertex $C$ and every rank 
$n-1$ vertex $A$, if $d(A,C)>1$ then there exists a 
vertex $L$ with $d(C,L)=1$ and $d(L,A)>2$.
\end{proposition}   

This proposition is an immediate consequence of Lemma \ref{l:le2} and the following result.

\begin{lemma} \label{l:ranky}
If $A<F_n$ is a free factor of rank $n-1$ and $C$ is a free factor of rank $1$ that is not contained in $A$, then
there exists a free factor $L$ of rank $2$ with $C<L$ and $L\cap A =1$.
\end{lemma}

\begin{proof} 
We may assume that $C=\<a_1\>$. We analyse $A$ according to the two cases in Corollary \ref{c:loops}. 
Suppose first that $\E_{a_2}(A)$ is finite, fix $M>\max\E_{a_2}(A)$ and let $L= \< a_1,\ a_2^Ma_1 a_3\>$.
Note that $L<F_n$ is a free factor, since $\<L,a_2\> = \<a_1,a_2,a_3\>$. 
A reduced word in the generators of $L$ either belongs to $C = \<a_1\>$ or else contains $a_2^M$ as a subword. 
The intersection of $C$ with $A$ is trivial, by hypothesis,  and reduced words of the latter form do not belong to $A$, by the definition of $M$,
so $L\cap A=1$.

It remains to consider the second case in Corollary \ref{c:loops}. Thus we assume now that 
$\core_*(A)$ is a tree with $n-1$ loops attached, labeled $a_2,\dots,a_n$. Observe that if $p$ is greater than the diameter of 
$\core_*(A)$, then $a_1^p\not\in \sub(A)$. It follows that no reduced word in the generators of the
rank $2$ free factor $L= \< a_1,\ a_2a_1^p a_3\>$ belongs to $A$. (Again, $L$ is a free factor because $\<L,a_2\> = \<a_1,a_2,a_3\>$.)
\end{proof}

\begin{proof}[Proof of Proposition \ref{rank i}] With Lemma \ref{1 and n-1} in hand, we compare
Lemma \ref{l:link-antip} with Proposition \ref{p:extend} to deduce that 
both $\FF_n(1)$ and $\FF_n(n-1)$ are preserved by every isometry of $\AF_n$. For $n=3$ there is nothing more
to prove, so we assume $n\ge 4$.
Let $A$ be a vertex of rank $i<n-1$
and let $V$ be a rank $(n-1)$ vertex with $A<V$.
The action of ${\rm{Aut}}(F_n)$ preserves the rank of vertices
and acts transitively on vertices of each rank, so by composing an arbitrary automorphism
$\psi\in{\rm{Isom}}(\AF_n)$ with a suitable element of ${\rm{Aut}}(F_n)$ we may assume that $\psi(V)=V$.
Then $\psi$ restricts to an isometry of $\Lk(V)\cong \AF_{n-1}$, and by induction on $n$ this 
restriction preserves the rank of vertices. 
\end{proof}

\section{Recognising Standard Apartments}\label{s:standard}

In the introduction we discussed the significance of {\em standard
  apartments}.

\begin{definition}
  A {\it standard apartment} in $\AF_n$
is the full subcomplex spanned by the free factors generated by the
non-empty proper subsets of a basis for $F_n$.
\end{definition}

For the second step
in our proof of Theorem \ref{thm1}, we must prove that every isometry of $\AF_n$ sends standard apartments to standard
apartments, i.e. the set of standard apartments is characteristic in the following sense.

\begin{definition}
We say that a collection of subcomplexes of a simplicial complex $X$ is {\it
characteristic} (or metrically distinguished) if it is preserved by the simplicial automorphism
group of $X$. 
\end{definition}

For example, for each $k$ the collection of $k$-simplices of $X$  will be characteristic.
In the previous section we proved that $\AF(i)$ is characteristic in $\AF_n$ for $i=1,\dots,n-1$. 
Our purpose in this section is to prove that the set of standard apartments is characteristic,
and a key step in the proof will be to show that the pairs of vertices $\{A,C\}$ with $\rank(A)=n-1,\
\rank(C)=1$ and $A\ast C=F_n$ is characteristic (the {\em Antipode Lemma}). Along the way, we
shall prove that various other types of subcomplexes are characteristic.

The Antipode Lemma is needed to distinguish standard apartments from {\em fake apartments} (as defined in Definition \ref{d:apartment}). Figure \ref{f:fake} illustrates two of the concerns that have to be overcome in the case $n=3$
and more elaborate fakes are discussed in Section \ref{s:fakes}.

\subsection{The Antipode Lemma}\label{s:antipode}
 
\begin{definition} A rank $n-1$ factor $\Lambda$ and a rank 1 factor $\<u\>$
are  {\em{algebraically antipodal}} if $\Lambda*\<u\>=F_n$. We write
$\Lambda\perp\<u\>$. 

$\L$ and $\<u\>$ are {\em{metrically antipodal}} in $\FF_n$ if $d(\<u\>,L)=2$
for all free factors $L$  
 with $d(\Lambda,L)=1$.
\end{definition}

\begin{remark} The condition that $\L$ and $\<u\>$ are metrically antipodal
  is equivalent to the following algebraic statement:
  $u\not\in\Lambda$ and for all free factors $L\subsetneq\Lambda$ there is a
  proper free factor of $F_n$ that contains both $L$ and $u$. We chose the more concise
  formulation in the definition because it makes clear that this
  property is invariant under isometries of $\FF_n$.
\end{remark}

\begin{theorem}[The Antipode Lemma] \label{p:antipode}
Let $\Lambda<F_n$ be a free factor of rank $n-1$ and $\<u\>$ a free factor of rank
  1. Then $\Lambda$ and $\<u\>$ are algebraically antipodal if and
  only if they are metrically antipodal.
\end{theorem}

\begin{proof}
 It follows easily from the definitions that algebraically antipodal
 implies metrically antipodal (Lemma \ref{l:link-antip}), so we will
 assume that $\<u\>\not\perp \Lambda$ and argue that $\L$ and $\<u\>$
 are not metrically antipodal. The case $u\in\Lambda$ is trivial, so
 suppose $u\not\in\Lambda$. By applying a suitable element of
 ${\rm{Aut}}(F_n)$ we may assume $\L=\<a_1,\dots,a_{n-1}\>$.  To
 complete the proof, it suffices to exhibit a free factor $L\subset\L$
 of rank $n-2$ such that $d(L,\<u\>)>2$. Our proof will show that if
 $f:\L\to \L$ is a fully irreducible automorphism and $L_0<\L$ is any
 free factor of rank $n-2$, then $L=f^k(L_0)$ has the desired
 property, provided $k>0$ is sufficiently large.

First we consider the case where no conjugate of $u$ is algebraically
antipodal to $\Lambda$. In this case, we argue using the subfactor projection 
$\pi_\Lambda$ described in section \ref{s:subfactor}. 
Consider $\pi_\Lambda([u])$. Choose $L$ (as above or
otherwise) so that the distance between $[L]$ and $\pi_\Lambda([u])$
is large; this is possible because $Lk([\Lambda])\cong \OF_{n-1}$ has infinite diameter,
using the modified definition of $\OF_2$ if $n=3$ (see Proposition \ref{l:loxo}).
The coarse Lipschitz property of $\pi_\Lambda$ (section \ref{s:subfactor}) tells us 
any short path
between $[L]$ and $[u]$ in $\OF_n$ must pass through a conjugacy class
of factors where $\pi_{\Lambda}$ is not defined. It follows that
there does not exist a free factor $B$ that contains both
$L$ and $\<u\>$, because $\pi_\Lambda[B]$ would be well-defined in that case, 
and $\pi_\Lambda[B]$ would be a distance at most $\delta$ (the constant of section \ref{s:subfactor})
from both   $\pi_{\Lambda}([u])$ and  $\pi_\Lambda([L])$. 
($B$ is not conjugate to $\Lambda$ because $L<B$ and $u\in B 
\smallsetminus\Lambda$,
whereas distinct conjugates of $\Lambda$ intersect trivially.)
 
  It remains to consider the case where $\<u\>\not\perp\Lambda$ but some conjugate of $\<u\>$ is antipodal to
  $\Lambda$. 
  By applying an 
  automorphism that fixes $\Lambda$ we may assume that, in reduced form, $u=wa_nw^{-1}$ where $w$ is a
  word whose first letter is $a_n^{\pm 1}$. Let $L=f^k(L_0)$ be as above and assume
  that $k$ is large enough to ensure that the injectivity radius of $\core(L)$ is at least  $2 |u|$ and 
  Corollary \ref{l:no-tiddler} holds for $\ell = 2|u|$
  with $L$ in the role of $A$ and $\Lambda$ in place of $F_n$.
  We will obtain a contradiction from the
  assumption that there is a free factor $B$ of rank $n-1$ that contains both $L$ and $u$.

First we observe that if there were such a factor, then $B=\<L,u\>$ and $\core_*(B)= \core_*(L)\vee \core_*\<u\>$.
To see this, note that if the canonical map $\core_*(L)\to\core_*(B)$ were not injective, then the fundamental group of the image would be a free factor $V\subseteq B$ that strictly contained $L$. As $\rk(B)=\rk(L)+1$, 
this would imply $V=B$. 
But the edges of the graph defining $V$ are labeled by letters from $\L$, whereas   $B$ is not contained in $\L$. Thus $\core_*(L)\to\core_*(B)$ is injective. As $B$ contains a conjugate of $a_n$ but not $a_n$ itself,  $\core_*(B)$ has a loop labeled $a_n$ based at a vertex $v\neq \ast $.  And since $wa_nw^{-1}\in B$, there is path from $*$ to $v$ labeled
$w$, which begins with an $a_n$-edge. As  $B$ has rank $\rk(L)+1$, this path is disjoint from $\core_*(L)$.
Thus $\core_*(B)= \core_*(L)\vee \core_*\<u\>$.  
 
Proposition \ref{p:identify} tells us 
that if $B=\<L,u\>$ were a free factor, then by identifying two vertices in $\core_*(B)$ we could obtain a graph $\G$
that  folded to  the standard rose $R_n$. We consider three cases, depending on the
location of the two vertices being identified, and reach a contradiction in
each case.

We shall refer to $\core_*\<u\>$ as a lollipop, with stalk labeled $w$ and loop $a_n$.

\smallskip

{\it Case 1: Suppose $v_0, v_1\in\core_*(L)$.}
In this case, the image of $\core_*(L)$ in ${\G}$ defines a free factor of  rank $n-1$ that contains $L$
and is contained in $\L$, hence is equal to $\L$. And $R_n$ is obtained by folding this image with
$\core_*\<u\>$, so $F_n = \L\ast \<u\>$, contrary to the assumption that $\L\not\perp\<u\>$.

\smallskip

{\it Case 2: Suppose $v_0\in\core_*(L)$ and $v_1\in\core_* \<u\>\smallsetminus \{ *\}$.}
In this case, an arc of the stalk of $\core_*\<u\>$ that contains $v_1$ but has no edges labeled $a_n$
might fold into $\core_*(L)$ which, by construction, has injectivity radius greater than $2|u|$. If $v_0$ were a distance at least $|u|$ from the
basepoint, then after this folding we would have a fully folded graph that still contained $\core_*(L)$. If $v_0$ is a distance less than $|u|$
from the basepoint,  let $\alpha$ be the label on the arc from $\ast$ to $v_0$, let $\tilde{\beta}$ be the prefix of $w$ labeling the arc from $\ast$ to $v_1$,
and let $\beta\in\Lambda$ be the word obtained from $\tilde{\beta}$ by deleting all occurences of $a_n$. Then $\<L, \alpha\beta^{-1}\>=\Lambda$ if
$\G$ folds to $R_n$, because $\<L, \alpha\beta^{-1}\>$ is the fundamental group of the graph obtained by collapsing the edges of $\G$ labeled $a_n$. But this contradicts
Corollary \ref{l:no-tiddler}, because $|\alpha\beta^{-1}|<2|u|$.

\smallskip  

{\it Case 3: Suppose $v_0, v_1 \in\core_*\<u\>$.}
We fold $\G_1 :=\core_*\<u\>/v_0\sim v_1$. If the initial edge on the stalk of the lollipop  $\core_*\<u\>$
is not identified with the loop of the lollipop during this folding, then $\core_*(L)\vee \fold(\G_1)$ is fully
folded and we are done. Otherwise, $\fold(\G_1)$ is the wedge of two loops, one labeled $a_n$ and the 
other either labeled by a word $c$ in the letters $a_i$ with $i\neq n$,
or else  labeled $c_1c_2c_3$, where $c_1$ and $c_3$ are non-empty words of this form
and $c_2$ is a non-empty word that begins and ends $a_n^{\pm 1}$. In the former case, 
we have a contradiction from Corollary \ref{l:no-tiddler}, because
 $\<L, c\> \subsetneq\Lambda$. In the latter case, the arcs labelled $c_1$ and $c_3$ fold into
 $\core_*(L)$   and the folding stops with $\core_*(L)$ still embedded.
 \end{proof}

\subsection{Apartments, fake and standard}

The barycentric subdivision  of the boundary of the standard $k$-simplex,  $\partial\Delta_k$,
 is the geometric realisation of the poset of nonempty 
proper subsets of ${\text{\bf k}}=\{0,1,\dots,k\}$ ordered by inclusion. The
barycentre of the face opposite $i\in  {\text{\bf k}}$ is ${\text{\bf k}}\smallsetminus \{i\}$.

\begin{definition}\label{d:apartment}
An {\em apartment} in $\FF_n$ is the image of a simplicial embedding $\sigma:\partial\Delta_{n-1}\hookrightarrow\FFn$ such that 
$\rank(\sigma(S)) = |S|$ for all $S\subset {\text{\bf n-1}}$. The
apartment is {\it fake} if it is not standard.
\end{definition}

Note that an apartment is standard if and only if its rank-1 vertices
form a basis for $F_n$.
Figure \ref{f:fake} illustrates two of the ways in which fake
apartments can arise. There are more examples in Section \ref{s:fakes}.

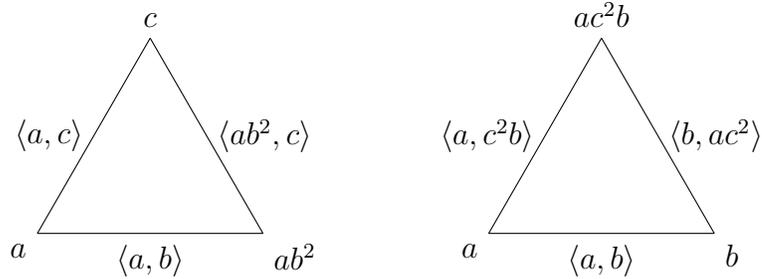
\begin{figure} [h]\label{f:fake}
\begin{tikzpicture}[scale=3]

\draw (0,0) -- (1,0);
\draw (0,0) -- (1/2,0.866);
\draw (1,0) -- (1/2,0.866);

\node[below left] at (0,0){$a$};
\node[below right] at (1,0){$ab^2$};
\node[above] at (1/2,0.866){$c$};

\node[below] at (1/2,0){$\<a,b\>$};
\node[right] at (3/4,0.433){$\<ab^2,c\>$};
\node[left] at (1/4,0.433){$\<a,c\>$};

\draw (2,0) -- (3,0);
\draw (2,0) -- (5/2,0.866);
\draw (3,0) -- (5/2,0.866);

\node[below left] at (2,0){$a$};
\node[below right] at (3,0){$b$};
\node[above] at (5/2,0.866){$ac^2b$};

\node[below] at (5/2,0){$\<a,b\>$};
\node[right] at (11/4,0.433){$\<b,ac^2\>$};
\node[left] at (9/4,0.433){$\<a,c^2b\>$};

\end{tikzpicture}
\caption{Two fake apartments in rank 3. In the first, the
  rank 2 factor $\<a,b\>$ is not generated by the adjacent rank 1
  factors. In the second, the rank 1 factors are not
  antipodal to the opposite rank 2 factors.} 
\end{figure}

\begin{lemma}\label{l:3}
An apartment in $\FF_3$ is standard if and only if each vertex is antipodal to the barycentre of the opposite face.
\end{lemma}

\begin{proof} The ``if" assertion is the non-trivial one. Suppose that the rank 1 vertices are $\<a\>,\, \<b\>,\, \<c\>$
and let $V$ be the barycentre of the face opposite $\<a\>$. Then $V=\<b,v\>$ for
some $v\in V$, and by hypothesis $F_3=V\ast\<a\>$. Thus $\<a,b\>$ is the unique free factor containing $a$ and $b$,
and it is therefore the barycentre of the face with vertices $\<a\>$ and $ \<b\>$. This is antipodal to
$\<c\>$, so $\{a,b,c\}$ is a basis for $F_3$.
\end{proof}

\begin{proposition}\label{p:standard} For $n\ge 3$,
every automorphism of $\FFn$ takes standard apartments to standard apartments.
\end{proposition}

\begin{proof} We proceed by induction on $n$. In the light of Lemma \ref{l:3}, the Antipode Lemma 
(Theorem \ref{p:antipode})  covers the
case $n=3$.  

Assume now that $n\ge 4$ and consider a rank $n-1$ vertex $V$ of a standard apartment $\sigma$ and let $\psi$
be an automorphism of $\FFn$. By composing $\psi$ with an element of ${\rm{Aut}}(F_n)$, we may assume
that $\psi$ fixes $V$. Then $\psi$  restricts to an automorphism of $\Lk(V)\cong\FF_{n-1}$, where by induction
we know that it takes standard apartments to standard apartments. The intersection $\sigma\cap\Lk(V)$ is
such an apartment, so the image under $\psi$ of its rank 1 vertices form a basis for $V$. The Antipode Lemma
tells us that the image under $\psi$ of the remaining rank 1 vertex of $\sigma$ is antipodal to $V$. Thus the image
under $\psi$ of the vertex set of $\sigma$ is a basis for $F_n$.
\end{proof}

\noindent{\bf{Notation.}} $\Delta(b_1,\dots,b_n)$ will denote the standard apartment associated to a
basis $\{b_1,\dots,b_n\}$ of $F_n$. A {\em face of rank $k$} is the
subcomplex  $\Delta[T]$ spanned by a $k$-element subset $T\subset\{b_1,\dots,b_n\}$. 
The face opposite $\Delta[T]$ is $\Delta[T^c]$, where $T^c=\{b_1,\dots,b_n\}\smallsetminus T$.
                  
\section{Sticks and propagation: the proof of Theorem \ref{thm1}} 

In this section we complete the proof of Theorem \ref{thm1}. 

\noindent{\bf{Summary of the proof.}}
 Given an automorphism $\Phi$ of $\AF_n$, with $n\ge 3$, we now know that $\Phi$ sends
 standard apartments to standard apartments. As ${\rm{Aut}}(F_n)$ acts transitively on the set
 of standard apartments, we can compose $\Phi$ with an element of ${\rm{Aut}}(F_n)$ 
 so as to assume that $\Phi$ leaves  a standard apartment
$\Delta=\Delta(a_1,\dots,a_n)$ invariant. The stabilizer of $\Delta$ in ${\rm{Aut}}(F_n)$ 
is the group of signed permutations $W_n \cong (\Z/2)^2\rtimes{\rm{sym}}(n)$ of the corresponding basis;
its action on $\Delta$ is the full group of rank-preserving symmetries of $\Delta$. By
composing   $\Phi$ with an element of $W_n<{\rm{Aut}}(F_n)$ we may assume that 
$\Phi$ fixes $\Delta$ pointwise. We would be done if this modification forced $\Phi$ to be
the identity on the whole of $\AF_n$, but it does not. For example,
automorphisms of the form $a_i\mapsto a_i^{\pm 1}$ fix $\Delta$ but
not $\AF_n$.

Let $\lambda$ be a Nielsen transformation for the basis $\{a_1,\dots,a_n\}$,
that is $[a_i\mapsto a_ia_j,\ a_k\mapsto a_k\ (k\neq i)]$ or $[a_i\mapsto a_ja_i,\ a_k\mapsto a_k\ (k\neq i)]$.
We say that $\lambda(\Delta)$ is {\em{Nielsen adjacent}} to $\Delta$; it has a large overlap with $\Delta$.

$\AF_n$ is the union of its standard apartments and
the index-2 subgroup of ${\rm{Aut}}(F_n)$  generated 
by Nielsen transformations acts transitively on the set of standard apartments. Thus,
by propagating to neighbours throughout $\AF_n$, we would be done if  any isometry of $\AF_n$ that fixed a standard apartment pointwise had to fix the Nielsen adjacent
apartments pointwise. Although this is not the case,  we shall see that standard apartments
have {\em canonical enlargements} that make this argument work: by composing $\Phi$ with a further element of
$W_n<{\rm{Aut}}(F_n)$ we can assume that it fixes the canonical enlargement of $\Delta$ and this
forces $\Phi$ to fix the canonical enlargement of each Nielsen adjacent apartment.

The vertices of these canonical enlargements are rank-1 vertices adjacent to $\Delta$ that
we call {\em sticks} and {\em supersticks}. 

\subsection{Sticks and snops}\label{s:sticks}

\begin{definition} The {\em sticks} at a face $\Delta[b_i,b_j]$ of rank $2$  in a standard apartment
$\Delta(b_1,\dots,b_n)$ are the  rank 1 factors of the form $\<b_i^\epsilon b_j^\delta\>,\ 
\epsilon,\delta=\{\pm 1\}$. 
\end{definition}

Note that this definition depends only on $\<b_i,b_j\>$ and  $\<b_i\>,\, \<b_j\>$, not
on the rest of $\Delta(b_1,\dots,b_n)$. There are  4 sticks at each rank $2$ face, so
$\Delta(b_1,\dots,b_n)$ has $4 {n\choose 2}$ sticks in total. See
Figure \ref{f:sticks}.  

\begin{lemma}\label{stick characterization} A rank 1 free factor $C<F_n$ is a stick
of the standard apartment $\Delta(b_1,\dots,b_n)$ if and only if,  for some $b_i\neq b_j$,
$d(C,\, \<b_i,b_j\>)=1$ and $C$ is antipodal to the 
barycentres of the rank $n-1$ faces opposite $\<b_i\>$ and $\<b_j\>$.
\end{lemma}

\begin{proof}  
This follows immediately from Lemma \ref{complement}.
\end{proof}

\begin{corollary}\label{c:sticks} The sets of sticks associated to standard apartments and their faces are
characteristic in $\AF_n$.
\end{corollary}

\begin{proof}
Immediate from Proposition \ref{p:standard} and the lemma.
\end{proof}

\begin{remark} As an indication of the way in which sticks determine the geometry of $\FF_n$ in a 
neighbourhood of an apartment, note that in $\FF_3$ each of the $12$
sticks of a standard apartment $\Delta(a,b,c)$ gives rise to a
2-sphere (after gluing in disks to each apartment) made from three
apartments: for example $bc$ determines the 2-sphere
$$\Delta (a,b,c)\cup \Delta (a,bc,b)\cup \Delta (a,bc,c).$$ 
The intersection of each pair of these spheres is $\Delta (a,b,c)$.
\end{remark}

\begin{remark} [Sticks and Cubes]  Our formal proofs for $n>3$ do not rely on the following 
description of sticks in terms of cubes, but nevertheless we include the general 
case in our discussion because  it provides useful insight into the local
geometry of $\FF_n$. 

The $4 {n\choose 2}$ sticks associated to a standard apartment  
parametrize the codimension-2 faces of an $n$-cube $I^n$. 
The signed permutations of the basis associated to the apartment form a subgroup
$W_n = (\Z/2)^2\rtimes{\rm{sym}}(n)<{\rm{Aut}}(F_n)$ and the action of this on the sticks
is the restriction of the standard representation of $W_n$ as the isometry group of the cube. Figure \ref{cube}
illustrates the case $n=3$.
\end{remark}  

There are 12 sticks associated to a standard apartment (if $n=3$) or
rank-3 face (if $n>3$). When three of these sticks lie in a common
free factor of rank $2$ in such a way that any two form a basis of the
subgroup they generate, we
say that these sticks form a {\em bonded triple}. We also say that two
sticks are {\em bonded} to each other if they lie in a common bonded
triple.  There are 8 bonded triples associated to each standard
apartment (if $n=3$) or rank-3 face (if $n>3$); they parametrize the
vertices of the cube in Figure \ref{cube}.

The 12 sticks also divide into 3 classes of {\em parallel sticks}, such that no pair of sticks in
a given class belong to the same bonded triple; these correspond to the 3 classes of parallel edges in Figure \ref{cube}.
Each parallelism class divides into 2 pairs: the {\em opposite} of a given stick is the one that 
labels the edge that is parallel but has no bonds in common.

With Corollary \ref{c:sticks} in hand, the following observation is immediate from these definitions.

\begin{lemma}
Isometries of $\AF_n$ preserve bonded triples and parallelism classes of sticks, as well as pairs of opposite sticks.
\end{lemma}

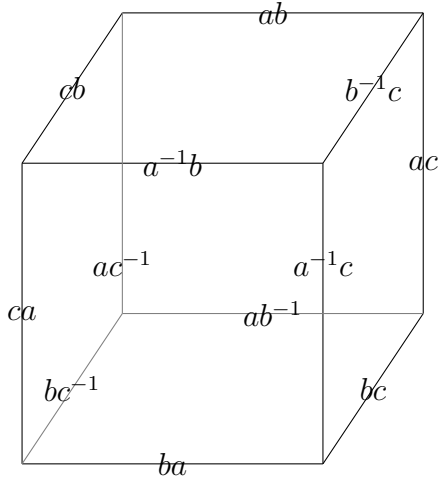
\begin{figure}[h]
\begin{tikzpicture}[scale=4]

\draw (0,0)--(1,0);
\draw (0,0)--(0,1);
\draw (0,1)--(1,1);
\draw (1,1)--(1,0);
\draw (0,1)--(1/3,3/2);
\draw (1,1)--(4/3,3/2);
\draw (1/3,3/2)--(4/3,3/2);
\draw (4/3,3/2)--(4/3,1/2);
\draw (1,0)--(4/3,1/2);
\draw[gray,thin] (0,0)--(1/3,1/2);
\draw[gray,thin] (1/3,1/2)--(4/3,1/2);
\draw[gray,thin] (1/3,1/2)--(1/3,3/2);

\node at (1/2,0) {$ba$};
\node at (1/2+1/3,3/2) {$ab$};
\node at (1/2,1){$a^{-1}b$};
\node at (1/2+1/3,1/2){$ab^{-1}$};
\node at (4/3,1){$ac$};
\node at (0,1/2){$ca$};
\node at (1,2/3){$a^{-1}c$};
\node at (1/3,2/3) {$ac^{-1}$};
\node at (7/6,1/4) {$bc$};
\node at (1/6,5/4) {$cb$};
\node at (1/6,1/4) {$bc^{-1}$};
\node at (7/6,5/4){$b^{-1}c$};

\end{tikzpicture}
\caption{The sticks associated to the standard apartment
  $\Delta = \Delta(a,b,c)$ parametrize the edges of the 3-cube. 
  Three sticks form a bonded triple
  (snop) when the three edges are adjacent to the same vertex.  The
  stabilizer of $\{a,b,c\}$ in ${\rm{Aut}}(F_3)$ is the full isometry
  group of the cube.}
\label{cube}
\end{figure}

\begin{figure}[h]
\begin{tikzpicture}[scale=4]
\draw (0,0)--(1,0);
\draw[rotate=60] (0,0)--(1,0);
\draw[shift={(1,0)},rotate=120] (1,0)--(0,0);

\draw[shift={(1/2,0)},rotate=-180/5] (0,0)--(1/2,0);
\draw[shift={(1/2,0)},rotate=-2*180/5] (0,0)--(1/2,0);
\draw[shift={(1/2,0)},rotate=-3*180/5] (0,0)--(1/2,0);
\draw[shift={(1/2,0)},rotate=-4*180/5] (0,0)--(1/2,0);

\draw[shift={(1/4,1.732/4)},rotate=60+180/5] (0,0)--(1/2,0);
\draw[shift={(1/4,1.732/4)},rotate=60+2*180/5] (0,0)--(1/2,0);
\draw[shift={(1/4,1.732/4)},rotate=60+3*180/5] (0,0)--(1/2,0);
\draw[shift={(1/4,1.732/4)},rotate=60+4*180/5] (0,0)--(1/2,0);

\draw[shift={(3/4,1.732/4)},rotate=-60+180/5] (0,0)--(1/2,0);
\draw[shift={(3/4,1.732/4)},rotate=-60+2*180/5] (0,0)--(1/2,0);
\draw[shift={(3/4,1.732/4)},rotate=-60+3*180/5] (0,0)--(1/2,0);
\draw[shift={(3/4,1.732/4)},rotate=-60+4*180/5] (0,0)--(1/2,0);

\node [below left] at (0,0) {$a$};
\node [below right] at (1,0) {$b$};
\node [above] at (1/2,1.732/2) {$c$};

\node [below] at (1/3,-1.732/4-0.05) {$ab$};
\node [above] at (1.1,0.8) {$b^{-1}c$};
\node [above] at (-0.1,0.8) {$ac$};

\end{tikzpicture}
\caption{The apartment $\Delta(a,b,c)$ in $\AF_3$ with its 12 sticks. The three
  sticks that are labeled  form a bonded triple (snop).} 
\label{f:sticks}
\end{figure}
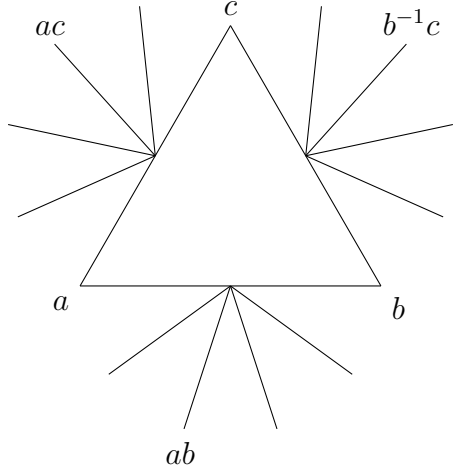

We have already noted that the sticks associated to a standard apartment $\Delta$  
parametrize the codimension-2 faces of a cube, and that in the case $n=3$ the vertices of the cube
correspond to bonded triples. In the general case, the vertices of this cube $I^n(\D)$ correspond to {\em snops},
which are defined as follows. (We shall not rely on this geometric description in our proofs.)

\begin{definition} A {\it snop}\footnote{Croatian for bundle.} is a collection $\mathcal B$ of
  sticks associated to a standard apartment $\Delta (b_1,\cdots,b_n)$
  with the following properties:
\begin{enumerate}[(1)]
\item Exactly one of the sticks associated to each rank-2 face $\Delta [b_i,b_j]$
  belongs to $\mathcal B$.
\item For every rank-3 face $\Delta [b_i,b_j,b_k]$, the 3 sticks in
  $\mathcal B$ form a {\it bonded triple}.
\end{enumerate}
\end{definition}

The following lemma is an immediate consequence of our previous results.

\begin{lemma}\label{l:snops}
Snops  are characteristic, i.e.
every isometry of $\AF_n$ takes snops to snops.
\end{lemma}

There are $2^n$ snops associated to a standard apartment $\Delta$. 
The 1-skeleton of the cube $I^n(\Delta)$ can be constructed by joining two snops with an edge if they share all
but $(n-1)$ of their sticks. (Distinct snops differ by at least $(n-1)$ sticks.)

The following proposition can be proved by analysing the faithful
action of the stabiliser of $\Delta(b_1,\dots,b_n)$ on the cube
$I^n(\Delta)$, arguing that if an isometry of the cube fixes
sufficiently many codimenson-2 faces then it must be the identity. We
leave the details of this proof to the reader and give a different
proof that adapts better to the case of $\OF_n$ considered in the next
section.

\noindent{\bf{Notation.}} The pointwise stabilizer in ${\rm{Aut}}(F_n)$ of the standard apartment 
$\Delta(a_1,\dots,a_n)$ is $(\Z/2)^n = \<\e_1,\dots,\e_n\>$ where $\e_i$ is the
involution that sends $a_i$ to $a_i^{-1}$ and fixes $a_j$ if ${j\neq i}$.

To be clear, when we say that an isometry {\em fixes} a subcomplex, we mean that it does so pointwise.

\begin{prop}\label{p:fix-sticks} If $\Phi\in{\rm{Isom}}(\FF_n)$ fixes $\Delta=\Delta(a_1,\dots,a_n)$,
then there exists $\theta\in \<\e_1,\dots,\e_n\>$ such that $\theta\circ\Phi$ fixes $\Delta$ and all of its sticks.
\end{prop}

We require a lemma. 

\begin{lemma}\label{l1}
If an isometry $\Phi$ of $\FF_3$ fixes the standard apartment $\Delta=\Delta(a,b,c)$ and a stick at $\Delta[a,b]$
then exactly one of $\{\Phi,\ \e_c\circ\Phi\}$ fixes $\Delta$ and all of its sticks.
\end{lemma}

\begin{proof}
The sticks at $\Delta[b,c]$ bonded to $\<ab\>$ are $\<b^{-1}c\>$ and $\<cb\>$, so if $\Phi$ 
fixes $\<ab\>$  then it must either exchange or
fix these sticks. Composing with $\e_c$ if necessary, we may assume that it fixes them. The
action of $\Phi$ as
an isometry of the cube in Figure \ref{cube} then fixes three edges of the top face. The only such isometry
is the identity.
\end{proof}

\noindent{\bf{Proof of Proposition \ref{p:fix-sticks}.}}  We shall proceed by induction. Suppose $n=3$
and consider a standard apartment $\Delta=\Delta(a,b,c)$ fixed by $\Phi$. If $\Phi$ does not fix the
stick $\<ab\>$ then we can compose $\Phi$ with an element of $\<\e_a, \e_b\>$ to arrange that it does. Then
Lemma \ref{l1} tells us that, composing with $\e_c$ if necessary, we may assume that $\Phi$ fixes 
all of the sticks of $\Delta$.

We now assume $n>3$ and consider a standard apartment $\Delta=\Delta(a_1,\dots,a_n)$ fixed by $\Phi$. 
Let ${\rm{Aut}}(F_{n-1}) \hookrightarrow{\rm{Aut}}(F_{n})$ be the subgroup fixing $a_n$ and acting in the
standard way on $\{a_1,\dots,a_{n-1}\}$. Consider the barycentre $V=\<a_1,\dots,a_{n-1}\>$ of the face
opposite $\<a_n\>$. We have $\Lk(V)\cong \AF_{n-1}$, where the isomorphism is ${\rm{Aut}}(F_{n-1})$-equivariant.
By induction, there exists $\theta\in\<\e_1,\dots,\e_{n-1}\>$ such that $\theta\circ\Phi$ fixes $\Delta$
and all of the sticks of $\Delta [a_1,\dots,a_{n-1}]$. Applying
Lemma \ref{l1} to $\Delta[a_1,a_2,a_n]$, we deduce 
that by further composing with $\e_n$ if necessary, we may assume that $\Phi$ fixes $\Delta$,
the sticks of $\Delta [a_1,\dots,a_{n-1}]$ and  the sticks of $\Delta[a_1,a_2,a_n]$. The remaining sticks
are based at $\Delta[a_i,a_n]\subset \Delta[a_1,a_i,a_n]$ with $i>2$. The sticks at $\Delta[a_1,a_i]\subset \Delta[a_1,a_i,a_n]$
are fixed by $\Phi$, as is $\<a_1a_n\>$. Moreover the latter is not fixed by $\e_n\circ\Phi$. So 
Lemma \ref{l1} tells us that $\Phi$ must fix all the sticks of $\Delta[a_1,a_i,a_n]$. This completes the induction.
\qed

\subsection{Supersticks and the end of the proof}

We obtain a more rigid framework of rank-1 vertices in the neighbourhood of an apartment by adding 
{\em supersticks} to sticks. In rank 3, the supersticks associated to an apartment are at distance 2 from
the apartment, but from  $n=4$ onwards they are adjacent to the barycentres of the rank-3 faces of the apartment.

\begin{definition} The {\em supersticks} associated to a standard apartment $\Delta (a_1,a_2,a_3)$ (if $n=3$)
or a rank 3 face  $\Delta [a_1,a_2,a_3]$ (if $n>3$) are the 24 rank 1 factors $\<a_i^{\delta_i} a_j^{\delta_j}a_k^{\delta_k}\>$
with $\{i,j,k\} = \{1,2,3\}$ and $\delta_i=\pm 1$.   
\end{definition}

\begin{lemma}\label{superstick characterization}  A rank 1 free factor of $F_3$ is a superstick
of the standard apartment $\Delta(a,b,c)$ if and only if it is antipodal to each of the rank 2 vertices of $\Delta(a,b,c)$.

For $n>3$, a rank 1 free factor $V<F_n$ is a superstick of the rank 3 face $\Delta[a,b,c]$ if and only if
$d(V,\<a,b,c\>)=1$ and $V$ is antipodal  in
$\Lk_{-}\<a,b,c\> \cong \AF_3$ to each of the rank 2 vertices of $\Delta(a,b,c)$.
\end{lemma}

\begin{proof}  
This follows immediately from Lemma \ref{complement}.
\end{proof}

\begin{corollary}\label{cc:sticks} The sets of supersticks associated to standard apartments and their
rank-3  faces are
characteristic in $\AF_n$.
\end{corollary}

We need one last lemma.

\begin{lemma}\label{l2} 
If an isometry $\Phi$ of $\FF_n$ fixes a standard apartment $\Delta=\Delta(a,b,c)$ (if $n=3$)
or rank-3 face $\Delta=\Delta[a,b,c]$ (if $n>3$) and it fixes the sticks of $\Delta$,
then it  also fixes all of the supersticks of $\Delta$. 
\end{lemma}

\begin{proof}
Consider first the superstick $\<abc\>$. As $M_1=\<a, bc\><F_n$ is the unique factor of rank 2 adjacent 
to both $\<a\>$ and $\<bc\>$, it must be fixed by $\Phi$. Likewise $M_2=\<c, ab\>$ must be fixed. 
The unique rank-1 factor adjacent to $M_1$ and $M_2$ is $\<abc\>= M_1\cap M_2$, so it too must be fixed by $\Phi$.
The general case is similar.
\end{proof}

\noindent{\bf{End of the Proof of Theorem \ref{thm1}}.} We refer the reader to the
summary of the proof given at the beginning of this section.
Given an automorphism $\Phi$ of $\AF_n$, with $n\ge 3$, we compose
it with an element of ${\rm{Aut}}(F_n)$ so as to assume that it leaves  a standard apartment
$\Delta=\Delta(a_1,\dots,a_n)$ invariant. We use Proposition \ref{p:fix-sticks} to compose $\Phi$ with a further
element of ${\rm{Aut}}(F_n)$ so that it fixes $\Delta$ and all of its sticks. Lemma \ref{l2} then
tells us that $\Phi$ fixes the supersticks of  $\Delta$. We will be done if we can argue that
this adjusted $\Phi$ fixes every standard apartment that is Nielsen adjacent to $\Delta$ and fixes all the sticks
(and hence supersticks) of such an apartment. 

Without loss of generality we may assume that the Nielsen transformation is $\lambda: a_1\mapsto a_1a_2$.
Consider  $\Delta_\lambda = \Delta(a_1a_2,a_2,\dots,a_n)$. The first point to observe is that 
every rank 1 vertex of $\Delta_\lambda$ is a vertex or stick of $\Delta$, and hence is fixed by $\Phi$.
Each vertex of $\Delta_\lambda$ is uniquely determined by its adjacent rank 1 vertices, so 
$\Phi$ must fix the whole of $\Delta_\lambda$.
The second point to observe is that every stick of $\Delta_\lambda$ is a vertex, stick or superstick of $\Delta$, with
the exception of the
sticks at $\Delta [a_1,\, a_1a_2]$. And since these last sticks are distinguished from one another by the 
sticks of $\Delta [a_1,\, a_1a_2, a_3]$ with which they form bonded triples, they too must be fixed.
\qed

\section{$\OF_n$ is rigid: Proof of Theorem \ref{thm2}}

Our proof of Theorem \ref{thm2} follows  the same outline of proof as Theorem \ref{thm1} but there are some
additional difficulties to be overcome in the case of $\OF_n$, particularly with regard to the recognition of 
standard apartments.  

We will typically write $[A]$ for the conjugacy class of a free factor $A<F_n$ but for rank-1 factors abbreviate $[\<u\>]$
to $[u]$, and often write $[a,b]$ for rank-2 factors.

\subsection{Step 1: Distinguishing the ranks of vertices.}\label{s:ostep1}

At various stages in the proof of Theorem \ref{thm1} we used  the isomorphism
$\Lk(A)\cong \FF_{n-1}$ for vertices of rank $n-1$ to facilitate induction arguments.
Lemma \ref{l:links} assures us that such arguments remain valid in $\OF_n$.

The following lemma can be established by choosing $L$ exactly
as in the proof of Lemma \ref{l:ranky}.

\begin{lemma} \label{l:oranky}
If $A<F_n$ is a free factor of rank $n-1$ and $C$ is a free factor of rank $1$, no conjugate of
which is contained in $A$, then
there exists a free factor $L$ of rank $2$ with $C<L$ such that no conjugate of $L$ intersects $A$ non-trivially.
\end{lemma} 

\begin{prop} For $n\ge 3$, 
every isometry of $\OF_n$ preserves the set of vertices of rank $i$,
for  $i=1,2,\cdots,n-1$.
\end{prop}

\begin{proof} The proof is a straightforward adaptation of the proof of Proposition \ref{rank i}.
To distinguish  vertices of rank $1$ or $n-1$ from those of rank $i$ with $1<i<n-1$, we 
prove that the former are not joins, and we do this by showing that they have diameter greater than $2$.
For $n=3$ there is nothing to prove, so we assume $n\geq 4$ and proceed by induction.
The link of a vertex of rank $n-1$ is isomorphic to $\OF_{n-1}$, which has infinite diameter (alternatively,
as in Lemma \ref{1 and n-1}, one can see easily that it has diameter at least $3$). For vertices of rank 1,
Lemma \ref{l:lk1} tells us that $\Lk_{\OF_n}([C])\cong \Lk_{\AF_n}(C)$, so the proof for $\FF_n$ applies directly.

The argument for  distinguishing  vertices of rank $n-1$ from vertices of rank $1$ also 
follows the case of $\AF_n$:
the proof of Lemma \ref{l:link-antip} shows that for every vertex $[A]\in\OF_n$ of rank $n-1$
there exist vertices $[C]$ of rank $1$ such that $d([A], [L]) =1$ implies $d([C], [L]) =2$, and
Lemmas \ref{l:oranky} and Lemma \ref{l:ole2} tell us this statement becomes false
if we reverse the roles of $A$ and $C$.

The inductive argument in the final paragraph of Section \ref{s:rank} remains
valid in the setting of $\OF_n$.
\end{proof}

\subsection{The Antipode Lemma}

\begin{definition} A rank $n-1$ vertex $[\Lambda]\in \OF_n$ and a rank 1 vertex $[u]\in \OF_n $
are  {\em{algebraically antipodal}} if there are factors $\Lambda_0\in [\Lambda]$
and $\<u^\gamma\>\in [u]$ such that  $\Lambda_0*\<u^\gamma\>=F_n$. We write
$[\Lambda]\perp [u]$. 

$[\L]$ and $[u]$ are {\em{metrically antipodal}} in $\OF_n$ if $d([u], [L])=2$
for all free factors $L$  
 with $d([\Lambda],[L])=1$.
\end{definition} 

\begin{theorem}[The Antipode Lemma]\label{O:antipode}
$[\Lambda]\perp [u]$  if and
  only if $[\L]$ and $[u]$ are {{metrically antipodal}} 
\end{theorem} 
 
\begin{proof} As was the case for $\AF_n$, it is easy to see that if $[\Lambda]\perp [u]$
then $[\L]$ and $[u]$ are {{metrically antipodal}}, and it is obvious that if $u$ is conjugate into $\Lambda$
then $[\L]$ and $[u]$ are not {{metrically antipodal}}. So what we must argue is that if no conjugate of
$u$ is contained in $\Lambda$ and  no conjugate of
$u$ is antipodal to $\Lambda$, then there is a free factor $L<\Lambda$ such that $d([\<u\>], [L])>2$.
This is what we proved  in the second paragraph of the proof of Theorem \ref{p:antipode}.
\end{proof}

\subsection{Step 2: Recognising Standard Apartments} 

The reader should compare the following definition to Definition \ref{d:apartment}. The more cumbersome
definition here reflects the fact that  in $\OF_n$ {\em apartments are not uniquely determined by their rank 1 vertices}.
This will cause us considerable difficulty, as will the fact that standard apartments  are difficult to characterise using the Antipode Lemma alone; see Example \ref{ex:bad} and Section \ref{s:fakes}.

\begin{definition}\label{d:O-apartment}
An {\em apartment} in $\OF_n$ is the image of a simplicial embedding 
$\sigma:\partial\Delta_{n-1}\hookrightarrow\OF_n$ such that 
$\rank(\sigma(S)) = |S|$ for all $S\subset {\text{\bf n-1}}$. The apartment is {\em standard} if 
it is the image under $\AF_n\to \OF_n$ of a standard apartment in $\AF_n$.
We shall maintain the notation $\Delta(a_1,\dots,a_n)$ for the standard apartment
associated to the basis $\{a_1,\dots,a_n\}$ and the notation $\Delta[T]$ for its faces;
if $|T|=k+1$ then $\Delta[T]$ is a {\em standard rank-$k$ face}.
\end{definition}

\begin{definition} [Sticks, supersticks, bonded triples] We define the
sticks, supersticks and bonded triples for standard faces in $\OF_n$  to be the images of the sticks, supersticks
and bonded triples in $\AF_n$.
For a standard apartment $\Delta=\Delta(b_1,\dots,b_n)$, the {\em{sticks of $\Delta$ at the rank-2 face 
$\Delta[b_i,b_j]$}}
are the rank 1 vertices of the form $[b_i^\epsilon b_j^\delta]$ (of which there are only two, because $b_i^\epsilon b_j^\delta$ and $ b_j^\delta b_i^\epsilon$ are conjugate and $[x]=[x^{-1}]$).
\end{definition}

\begin{remark}\label{r:sticks}
(1) Considerable care is needed with this definition: the  {\em ``sticks of $\Delta$ at the face $\Delta[b_i,b_j]$" }
depend on 
$\Delta$ and not just $\Delta[b_i,b_j]$ and its neighbours $[b_i], [b_j]$.
 Indeed, if one drops the reference to $\Delta$ then there are infinitely many sticks at
$\Delta[b_i,b_j]$. To see this note, for example, that for any $u\in \<a_1, a_2\>$, the triple $[ua_1u^{-1}, a_2 ],\, [ua_1u^{-1}],\, [a_2]$ is identical to
$[a_1,a_2],\, [a_1],\, [a_2]$, but the sticks of $\Delta (ua_1u^{-1}, a_2,\dots,a_n)$ at $[ua_1u^{-1}, a_2 ]=[a_1,a_2]$ are $[ua_1u^{-1}a_2]$ and $[ua_1^{-1}u^{-1}a_2]$,
whereas the sticks of $\Delta (a_1, a_2,\dots,a_n)$ at $[ua_1u^{-1}, a_2 ]=[a_1,a_2]$ are $[a_1a_2]$ and $[a_1^{-1}a_2]$.
\smallskip

(2) As $\Delta=\Delta(b_1,\dots,b_n)$ has only two sticks at $\Delta[b_i,b_j]$,  
it has $2 {n\choose 2}$ sticks in total. There are
8 supersticks associated to each 
standard apartment (if $n=3)$ or rank 3 face (if $n>3$). 
\smallskip

(3) It is no longer useful to discuss
which pairs of  sticks are bonded, because any pair of sticks associated to a rank 3 face will be bonded,
but it remains true and useful that any two sticks in a bonded triple uniquely define the third.
\smallskip

Passing to conjugacy classes  $A\mapsto [A]$ preserves the relation of being algebraically antipodal, 
so sticks of an apartment remain antipodal to the barycentres of opposite faces. 
But at this stage we do not have a metric
characterisation of sticks   (as in Lemma \ref{stick characterization})  
because we do not yet know that
isometries of $\OF_n$ take standard apartments to standard apartments. 
\end{remark}

\begin{figure}[h]
\begin{tikzpicture}[scale=4]
\draw (0,0)--(1,0);
\draw[rotate=60] (0,0)--(1,0);
\draw[shift={(1,0)},rotate=120] (1,0)--(0,0);

\draw[shift={(1/2,0)},rotate=-2*180/5] (0,0)--(1/2,0);
\draw[shift={(1/2,0)},rotate=-3*180/5] (0,0)--(1/2,0);

\draw[shift={(1/4,1.732/4)},rotate=60+2*180/5] (0,0)--(1/2,0);
\draw[shift={(1/4,1.732/4)},rotate=60+3*180/5] (0,0)--(1/2,0);

\draw[shift={(3/4,1.732/4)},rotate=-60+2*180/5] (0,0)--(1/2,0);
\draw[shift={(3/4,1.732/4)},rotate=-60+3*180/5] (0,0)--(1/2,0);

\node [below left] at (0,0) {$a$};
\node [below right] at (1,0) {$b$};
\node [above] at (1/2,1.732/2) {$c$};

\node [below] at (1/3,-1.732/4-0.05) {$ab$};
\node [above] at (1.1,0.8) {$bc$};
\node [above] at (-0.1,0.8) {$ac$};

\node [below] at (2/3,-1.732/4-0.05) {$a^{-1}b$};
\node [below] at (1.3,0.55) {$b^{-1}c$};
\node [below] at (-0.3,0.55) {$a^{-1}c$};

\end{tikzpicture}
\caption{The apartment $\Delta(a,b,c)$ in $\OF_3$ with its 6 sticks.} 
\label{sticks}
\end{figure}
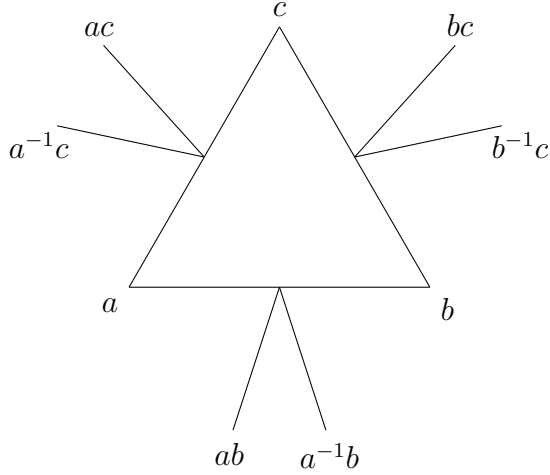

\begin{example}\label{ex:bonds} $\Delta(a, b,c )$ has four bonded triples (snops) in $\OF_3$
\begin{itemize} 
\item $ab, b^{-1}c, ac$,
\item $ab, bc, a^{-1}c$,  
\item $a^{-1}b, bc, ac$,
\item $a^{-1}b, b^{-1}c, a^{-1} c$.
\end{itemize}
For any pair of sticks chosen from two rank 2 faces, there is a
unique stick at the third face that  forms a bonded triple (snop) with that pair.

The eight supersticks of $\Delta(a, b,c )$ in $\OF_3$ are
\begin{itemize} 
\item  $abc,\ abc^{-1},\  ab^{-1}c,\  ab^{-1}c^{-1},\  acb,\ acb^{-1},\  ac^{-1}b,\ ac^{-1}b^{-1}c.$
\end{itemize}
\end{example}

\begin{example}\label{ex:bad}
We describe an example of a fake (i.e. non-standard) apartment of $\OF_3$ in
which all pairs of opposite vertices are antipodal.  

Starting with the standard apartment $\Delta(a,b,c)$, we replace $[a,b]$ by
$[a, \gamma b \gamma^{-1}]$ with $\gamma=baca^{-1}$ to obtain the face $\Delta'$.
 The graph $\core \<a, \gamma b \gamma^{-1}\>$
consists of two loops labeled $a,b$ joined by an arc labeled $baca^{-1}$. To
see that $[a, \gamma b \gamma^{-1}]$ is antipodal to $[c]$, we glue a loop labeled $c$ to one of the
endpoints of the edge of $\core \<a, \gamma b \gamma^{-1}\>$ labeled $c$ and 
 fold to obtain the rose $R_3$. To see that the apartment $\Delta'$ is fake, observe that it has
 no sticks at $[a, \gamma b \gamma^{-1}]$: more precisely, 
 there are no rank 1 factors $[u]$ adjacent to $[a, \gamma b \gamma^{-1}]$
 that are antipodal to both $[a,c]$ and $[b,c]$. Indeed, any cyclically reduced word in the conjugacy class $[u]$
 must label a loop in  $\core \<a, \gamma b \gamma^{-1}\>$ that is not $a^{\pm 1}$
or $b^{\pm 1}$, and the label on any such loop contains more than one occurence of both $a$ and $b$, so is not antipodal
to $[a,c]$ or $[b,c]$.
\end{example}

Fortunately, 
the problem identified in this example is the only new obstruction to recognising standard apartments in rank $3$.

\begin{prop}\label{apartments}
Let $\Delta$ be an apartment in $\OF_3$ and assume
\begin{enumerate}[(1)]
\item opposite vertices of $\Delta$ are antipodal, and
\item $\Delta$ has ``a potential stick" at each rank 2 vertex, i.e. there is an adjacent rank 1 vertex
that is antipodal to the other two rank 2 vertices of $\Delta$.
\end{enumerate}
Then $\Delta$ is a standard apartment.
\end{prop}

\begin{proof} Let $\{a,b,c\}$ be a basis for $F_3$.
We may assume that $\Delta$ has opposing vertices $[a]$ and $[b,c]$.
By applying an automorphism of $F_3$ that fixes $a$ and leaves $\<b,c\>$
invariant, we can assume that one of the rank 1 vertices adjacent to $[b,c]$ is $[b]$.
The rank 2 vertex $V$ between $[a]$ and $[b]$ is then $[b, \gamma a \gamma^{-1}]$ for some 
$\gamma\in F_3$. If $[u]$ is a potential stick of $\Delta$ at $[b, \gamma a \gamma^{-1}]$, then it is antipodal
to $[b,c]$ and hence the cyclically reduced form of $u$ contains exactly one occurence of $a$, by Lemma \ref{complement}. 
This word labels a tight (i.e. locally-injective) loop in $\core\<b, \gamma a \gamma^{-1}\>$.
The only tight loops
with a single occurence of $a$ in their label, besides  $a^{\pm 1}$, are the loops labeled $a^{\pm 1}\gamma^{-1}b^p\gamma$
with $p\neq 0$, and these only qualify if there is no occurence of $a$ in $\gamma$. 
The loops  $a^{\pm 1}$ can be excluded as potential sticks because they are not antipodal to the rank-2 vertex opposite
$[b]$, since that already contains a conjugate of $a$. Thus the existence of a potential 
stick at $V$ forces $\gamma\in\<b,c\>$,
and after applying the automorphism that fixes $b$ and $c$ and sends $a\mapsto \gamma^{-1}a \gamma$ we may assume
$V=[a,b]$. 

Consider now the  rank 1 vertex of $\Delta$ opposite $V$; call it $[x]$. Since $[x]$ is antipodal to $[a,b]$, the cyclically reduced word 
conjugate to $x$ contains exactly one $c^{\pm 1}$, and since
$x$ is conjugate into $\<b,c\>$ we may assume (by conjugating and replacing $x$ with $x^{-1}$)
that $x=b^mc$ for some $m$. After applying the
automorphism that fixes $a,b$ and sends $c\mapsto b^{-m}c$, we have
$x=c$. Then $\Delta$ has 5 of its vertices in common with the standard
apartment $\Delta(a,b,c)$, and the last one is the conjugacy class of a factor of the form
$H=\<a, \delta c \delta^{-1}\>$. The labeled graph $\core(H)$ has
loops $a$ and $c$ connected by an arc $\delta$. Repeating the argument used to analyse $V$,
we see that $H$ can only contain a rank 1 factor antipodal to $[b,c]$ if $\delta$ contains no $c$, and it
can only contain a rank 1 factor antipodal to $[a,b]$ if $\delta$ contains no $a$. Thus  
$\delta=b^q$, and  the automorphism of $F_3$ that fixes $a$ and $b$ and
sends $c\mapsto b^{-q} c b^q$ will map $\Delta$ to the  standard apartment $\Delta[a,b,c]$.
\end{proof}

\begin{corollary}\label{n=3 good}
Isometries of $\OF_3$ take standard apartments to standard apartments.
\end{corollary}

\begin{proof} 
In Step 1 (section \ref{s:ostep1})
we proved that isometries of $\OF_3$ preserve rank,  in the Antipode Lemma we proved that 
they send antipodal pairs to antipodal pairs, and in Proposition \ref{apartments} we characterised standard
apartments in terms of these invariants.
\end{proof}

The last lemma we need before concluding that isometries preserve standard apartments is the following. 
The fake apartments described in Section \ref{s:fakes} illustrate the need for condition (3) in this lemma.

\begin{lemma} \label{l:build-up}
Let $n\ge 3$. An apartment $\Delta$ in $\OF_n$ is standard if and only if it
satisfies the following conditions:
\begin{enumerate}
\item Each rank $(n-1)$ face of $\Delta$ is standard.
\item Every rank 1 vertex of $\Delta$ is antipodal to the barycentre of the opposite face.
\item Adjacent to each rank $(n-1)$ vertex $V$ of $\Delta$, there is a rank $1$ vertex that
is antipodal to every  rank $(n-1)$ vertex  of $\Delta$ other than $V$.
\end{enumerate}
\end{lemma}

\begin{proof} First note that standard apartments satisfy these conditions: for (3), a suitable
rank $1$ factor adjacent to $[a_1,\dots,a_{n-1}]\in \Delta(a_1,\dots,a_n)$  is $[a_1\dots a_{n-1}]$.

For the converse,
Proposition \ref{apartments} covers the case $n=3$,  
so we suppose $n>3$. Condition (1) lets us assume that there is a basis $\{a_1,\dots,a_n\}$ of $F$
such that one of the codimension-1 faces of $\Delta$ is the standard $\Delta[a_1,\dots,a_{n-1}]$.
Condition (2) says that the rank 1 vertex opposite this face is $[x]$ where $x$ is antipodal to $\<a_1,\dots,a_{n-1}\>$.
The action of ${\rm{Aut}}(F_n)$ (through
${\rm{Out}}(F_n)$) preserves conditions (1), (2) and (3), so we are free to
apply an automorphism to ensure that $x=a_n$. 

Consider the codimension 1 face $Y_1$ of $\Delta$ opposite $[a_1]$.  By condition (1),
this is standard, so the barycentre of the face is $[V_1]$ where $V_1$ is generated by $\<a_2,\dots,a_{n-1}\>$ and 
a conjugate of $a_n$, say $a_n^{\gamma_1}$.
We can assume that $\gamma_1\in \<a_1,\dots,a_n\>$ is a word that does
not end in $a_n^{\pm 1}$ and (if nontrivial) starts with $a_n^{\pm 1}$
-- it is the label on the
bridge of $\core(V_1)$ connecting the rose with petals $a_2,\dots,a_{n-1}$ to the loop labeled $a_n$.
For $j=2,\dots,n-1$, the barycentre of the edge of $Y_1$
joining $[a_j]$ to $[a_n]$ is $[a_j, a_n^{\gamma_1}]$. Because of our
assumptions on $\gamma_1$, the core graph of $[a_j, a_n^{\gamma_1}]$
consists of loops labeled $a_j$ and $a_n$ with the bridge connecting
them with the label precisely $\gamma_1$.

Similar considerations apply to face $Y_i$ opposite $[a_i]$ for $i=2,\dots,n-1$ and
we define $V_i$ and $\gamma_i$ accordingly. For example, $V_2=\<a_1,a_3,\dots,a_{n-1}, a_n^{\gamma_2}\>$
and for $j=1,3,\dots,n-1$, the barycentre edge of $Y_2$ joining
$[a_j]$ to $[a_n]$ is $[a_j, a_n^{\gamma_2}]$. 

The edge  joining $[a_3]$ to $[a_n]$ in $Y_1$ is, of course, the same as the edge joining them in $Y_2$, so 
$[a_3, a_n^{\gamma_1}] = [a_3, a_n^{\gamma_2}]$. Comparing core graphs, we conclude that 
$\gamma_1=\gamma_2$ since
both are the label on the bridge. 
Proceeding in this manner, we conclude that $\gamma_i=\gamma_j$ for all $i,j\in\{1,\dots,n-1\}$.
If this common conjugator  $\gamma$ lies $ \<a_1,\dots,a_{n-1}\>$, then
 the automorphism that fixes $a_i$ for $i<n$ and conjugates $a_n$ by  $\gamma^{-1}$
will map $\Delta$ to the standard apartment $\Delta(a_1,\dots,a_n)$, so $\Delta$ is standard. 

To complete the proof, we argue that if $\gamma\not\in\<a_1,\dots,a_{n-1}\>$ then $\Delta$ would not satisfy condition (3). 
If a rank 1 vertex $[u]$ is adjacent to $[V_1]$, there is a reduced loop in $\core(V_1)$ labeled $u$.
The key point to note is that each reduced  loop in $\core(V_1)$ either lies in the rose with labels $a_2,\dots,a_{n-1}$, or 
runs only around the loop labeled $a_n$, or else traverses the bridge labeled $\gamma$ twice. In the first
case $[u]$ is not antipodal to $[a_1,\dots,a_{n-1}]\in\Delta$, in the second case it is not 
antipodal to $[V_2]$, and in the last case every conjugate of $u$
contains at least 3 occurrences of the letter $a_n$, so  $[u]$
is not antipodal to $[a_1,\dots,a_{n-1}]$, by Lemma \ref{complement}. 
\end{proof}

\begin{prop}\label{O:standards} For $n\ge 3$, every isometry of $\OF_n$ takes standard apartments to 
standard apartments. 
\end{prop}

\begin{proof} Same as Corollary \ref{n=3 good}.
\end{proof}

\begin{corollary}\label{o:sticks}
For $n\ge 3$, every isometry $\Psi$  of $\OF_n$ takes the sticks of a standard apartment $\Delta$
to the sticks of $\Psi(\Delta)$.
\end{corollary} 

\begin{proof} It follows from Lemma \ref{complement} that
the sticks of $\Delta$ at 
a rank $2$ face $\Delta[a,b]$ are the unique rank 1 vertices $V$ adjacent to $[a,b]$
with the property that  for every rank $3$ face $\Delta[a,b,c]$, the vertex
$V$ is antipodal to $[a,c]$ and $[b,c]$ in $\Lk_-([a,b,c])\cong\OF_3$.
And $\Psi$ transports this condition to the sticks of the
standard apartment $\Psi(\Delta)$.
\end{proof}

Similarly,  following Lemma
\ref{superstick characterization} we have:

\begin{corollary}\label{o:ssticks}
For $n\ge 3$, every  isometry $\Psi$  of $\OF_n$ takes the supersticks of a standard apartment $\Delta$
to the supersticks of $\Psi(\Delta)$. 
\end{corollary}

\subsection{The endgame}

The sum of our previous results tells us that for $n\ge 3$, every isometry of $\OF_n$ maps
standard apartments to standard apartments, respecting their sets of sticks, supersticks and bonded triples
(triples of sticks contained a common factor of rank $2$).
We shall deduce Theorem \ref{thm2} by following the final steps in the proof of Theorem \ref{thm1}; only
minor adjustments are needed, except for the issue resolved in Lemma \ref{o:l2}.

It will be convenient to consider the action of ${\rm{Aut}}(F_n)$ on $\OF_n$ (with the inner automorphisms
acting trivially), as the subgroups ${\rm{Aut}}(F_{n-1})\hookrightarrow {\rm{Aut}}(F_n)$ fixing basis elements
appear in the proof.
The pointwise stabilizer in ${\rm{Out}}(F_n)$ of the standard apartment 
$\Delta(a_1,\dots,a_n)$ is $(\Z/2)^n = \<\e_1,\dots,\e_n\>$ where $\e_i$ is the
involution that sends $a_i$ to $a_i^{-1}$ and fixes $a_{j\neq i}$. The diagonal element
$\iota = \e_1\dots\e_n$ will play a special role, related to  the following observation.

\begin{lemma}\label{l:iota-sticks}
$\iota$ acts trivially on the set of sticks associated to the standard apartment $\Delta(a_1,\dots,a_n)$ in $\OF_n$,
but it acts without fixed points on the set of supersticks.
\end{lemma}

\begin{prop}\label{o:fix-sticks} If $\Phi\in{\rm{Isom}}(\OF_n)$ fixes $\Delta=\Delta(a_1,\dots,a_n)$,
then there exists $\theta\in \<\e_1,\dots,\e_n\>$ such that $\theta\circ\Phi$ fixes $\Delta$ and all of its sticks.
\end{prop}

The inductive proof of Proposition \ref{p:fix-sticks} applies verbatim to this proposition (replacing 
$\AF_{n-1}$ with $\OF_{n-1}$) once we have
the following analogue of Lemma \ref{l1} in hand.
 
\begin{lemma}\label{o:l1}
If an isometry $\Phi$ of $\OF_3$ fixes the standard apartment $\Delta=\Delta(a,b,c)$ and the sticks at $\Delta[a,b]$,
then one exactly one of $\{\Phi , \e_c\circ\Phi\}$ fixes $\Delta$ and all of its sticks.
\end{lemma}

\begin{proof}  If $\Phi$ exchanges the two sticks at $[b,c]$ then we compose with  
 $\e_c$ so  that it fixes them. It must then fix the sticks at $[a,c]$,  because they are contained in
 bonded triples where the other two sticks are fixed, and each pair of sticks in a triple uniquely determines
 the third stick (see Example \ref{ex:bonds}).
\end{proof}
 
At this stage in the proof of Theorem \ref{thm1} we argued (Lemma \ref{l2})
that if an isometry $\Phi$ of $\FF_n$ fixes a standard apartment $\Delta$ and its sticks,
then it  also fixes all of the supersticks of that apartment.  This is not true in the case of $\OF_n$; it
has to be adjusted as follows. 

Note that since $\iota$ acts freely on the supersticks of 
$\Delta(a_1,\dots,a_n)$, the word ``one" in the following statement means ``exactly one".

\begin{lemma}\label{o:l2} For $n\ge 3$, 
if an isometry $\Phi$ of $\OF_n$ fixes the standard apartment $\Delta(a_1,\dots,a_n)$
and its sticks, then one of $\Phi$ and $\iota\circ\Phi$ fixes the apartment, its sticks and its supersticks. 
\end{lemma}

\begin{proof} The vertices $[M_i]$ and $[M_i']$ appearing in this proof should be regarded as {\em midpoints}
between the rank 1 vertices of $\Delta=\Delta(a_1,\dots,a_n)$ and the sticks of $\Delta$; these midpoints come in pairs.

Consider first the superstick $[a_1a_2a_3]$. The rank 2 vertices $[M]$ adjacent to both $[a_1a_2]$ and $[a_3]$
in $\OF_n$ are of the form $[a_1a_2, a_3^\gamma]$ or $[a_2a_1, a_3^\gamma]$, where $\gamma$
is the label on the arc in $\core(M)$ connecting the loop labeled $a_1a_2$
or  $a_2a_1$ to the loop labeled $a_3$.  
The key point to observe is that if $\gamma\neq 1$ then $\core(M)$ does not contain a loop labeled by a superstick
$a_i^{\delta_1}a_j^{\delta_2}a_k^{\delta_3}$ with $|\delta_1|=|\delta_2|=|\delta_3|=1$. Thus the only rank 2 vertices
$[M]\in\OF_n$ at distance $1$ from $[a_1a_2]$ and $[a_3]$ and a superstick of $\Delta$ are 
$[M_1]:=[a_1a_2, a_3]$ and $M_1':=[a_2a_1, a_3]$. Likewise, 
the only rank 2 vertices
$[M]\in\OF_n$ at distance $1$ from $[a_2a_3]$ and $[a_1]$ and a superstick of $\Delta$ are 
$[M_2]:=[a_2a_3, a_1]$ and $M_2':=[a_3a_2, a_1]$. 

The two supersticks carried by $M_1$ are $[a_1a_2a_3]$ and $[a_1a_2a_3^{-1}]$, while 
$M_1'$ carries  $[a_1a_3^{\pm 1}a_2]$ and  $M_2$ carries  $[a_1(a_2a_3)^{\pm 1}]$ and  $M_2'$ carries  
$[a_1 (a_3a_2)^{\pm 1}]$. Thus $M_1$ and $M_2$ have a single superstick in common, as do $M_1'$ and $M_2'$,
and no other combination does.  

As $\Phi$ fixes $[a_1a_2]$ and $[a_3]$ and takes supersticks to supersticks, it must fix both of $M_1$ and $M_1'$
or interchange them. Likewise it must fix both of $M_2$ and $M_2'$
or interchange them. And if it interchanges $M_1$ and $M_1'$ then it must also interchange $M_2$ and $M_2'$,
since $M_1$ and $M_2$ have a superstick in common, whereas $M_1$ and $M_2'$ do not.
The action of $\iota$ fixes $\Delta$ and its sticks while interchanging $M_1$ and $M_1'$ and  interchanging
$M_2$ and $M_2'$. So by composing with $\iota$ if necessary, we may assume that $\Phi$ fixes 
each of $M_1, M_1',M_2, M_2'$. It must then also fix the common supersticks that pairs of these factors support,
and the remaining supersticks that they carry must then also be fixed. Thus $\Phi$ (possibly adjusted by $\iota$)
must fix all six of the sticks listed above. The remaining supersticks of $\Delta[a_1,a_2,a_3]$ are
$[a_1a_2^{-1}a_3]$ and $[a_1a_2^{-1}a_3^{-1}]$. These too must be fixed because the latter is supported in common
with $[a_1a_3a_2]$ on a midpoint graph between $[a_1]$ and $[a_3a_2]$, whereas the former is not.

At this point we are done in the case $n=3$, but to 
complete the proof of the lemma in the general case
we must argue that because $\Phi$ fixes the supersticks associated to
one rank-3 face, it fixes the supersticks on all rank 3 faces. The argument given above shows that
$\Phi$ either fixes all or none of the supersticks at a rank 3 face, so it will be enough to prove that 
$\Phi$ fixes one of the supersticks at an adjacent face; we focus on $[a_1a_2a_4]$.

Observe that $V=[a_1a_2, a_3,a_4]$ is the unique rank-3 vertex adjacent to $M_1,\, [a_3],\, [a_4],\, [a_3a_4]$,
all of which we know to be fixed by $\Phi$, and $\core(V)$ is the wedge of loops labeled $a_3, a_4, a_1a_2$.
The only superstick of $\Delta[a_1, a_2, a_4]$ carried by this graph is $[a_1a_2a_4]$; in other words
this is the only such superstick that is a distance $1$ from $V$. Thus the isometry $\Phi$ must fix $[a_1a_2a_4]$. 
\end{proof}

The proof of the following observation is contained in the preceding proof.

\begin{addendum}\label{addend} If $\Phi$ fixes $\Delta(a_1,\dots,a_n)$, all of its sticks, and all of its substicks, 
then, for all 
distinct triples $i,j,k\in\{1,\dots,n\}$,
it also fixes each of the rank 2 vertices $[M]$ adjacent to both $[a_i]$ and $[a_j,a_k]$ 
\end{addendum}

We need one last lemma.
 
\begin{lemma}\label{one off} 
Let $\D_1=\Delta(b_1,b_2,\cdots,b_n)$ be a standard apartment that contains all
the vertices of $\D_0=\Delta(a_1,a_2,\cdots,a_n)$ except for
$\<a_2,a_3,\cdots,a_n\>$.
\begin{enumerate}[(1)]
\item If $n>3$ then $\D_1=\D_0$.
\item If $n=3$ then $\D_1=\D_0(a_1,a_2^{a_1^k},a_3)$ for some $k\in\Z$.
\end{enumerate}
\end{lemma}

\begin{proof} The rank $n-1$ factor $V$ that $\D_1$ has in place of $\<a_2,\cdots,a_n\>$
 contains, up to conjugation, both
$\<a_2,\cdots,a_{n-1}\>$ and $\<a_3,\cdots,a_n\>$. For $n>3$ this implies
that $V$ is conjugate to $\<a_2,\cdots,a_n\>$, by Lemma \ref{lift}.

If $n=3$ then $V$ must have the form $V=\<a_2^\gamma,a_3\>$, whose core graph has
two loops, labeled $a_2$ and $a_3$ connected by an arc labeled
$\gamma$. 
Arguing with the existence of sticks (as in Example \ref{ex:bad}) we see that $\gamma$ must be a power of $a_1$.
\end{proof}

\noindent{\bf{End of the Proof of Theorem \ref{thm2}}.}  
Given an automorphism $\Phi$ of $\OF_n$, with $n\ge 3$, we compose
it with an element of ${\rm{Aut}}(F_n)$ so as to assume that it leaves  a standard apartment
$\Delta=\Delta(a_1,\dots,a_n)$ invariant. We use Proposition \ref{o:fix-sticks} to compose $\Phi$ with a further
element of ${\rm{Aut}}(F_n)$ so that it fixes $\Delta$ and all of its sticks. Lemma \ref{o:l2} then
tells us that, after composing with $\iota$
if necessary, $\Phi$ fixes the supersticks of  $\Delta$. We will be done if we can argue that
this adjusted $\Phi$ fixes every standard apartment that is Nielsen adjacent to $\Delta$ and fixes all the sticks
and supersticks of such an apartment. 

Without loss of generality we may assume that the Nielsen transformation is $a_1\mapsto a_1a_2$.
Consider  $\Delta_\lambda = \Delta(a_1a_2,a_2,\dots,a_n\}$.  
Every rank 1 vertex of $\Delta_\lambda$ is a vertex or stick of $\Delta$, and all of the faces that
do not include the vertex $[a_1a_2]$ are fixed as they lie in $\Delta$. 
Proceeding by induction on the rank we may 
assume that every vertex except $V=\<a_1a_2,a_3,\cdots,a_n\>$ is fixed.  
It then follows from Lemma \ref{one off} that $V$ is also fixed.
All the sticks of this apartment except for one (namely
$a_1a_2a_2$) are either vertices, sticks or supersticks of
$\Delta$, so all of the sticks are fixed. It follows from Lemma \ref{o:l2} that 
$\Phi$ fixes all of the supersticks of $\Delta_\lambda$ or none of them (because $\iota$
acts without fixed points on the set of supersticks). But there is one
that we know it does fix, namely $[a_1a_2^2a_3]$, because this is the only superstick of $\Delta_\lambda$
that is carried by the rank 2 vertex $[a_2, a_1a_3]$, and this is one of the midpoint vertices $[M]$ that
Addendum \ref{addend} tells us is fixed by $\Phi$. This completes the proof.
\qed

\section{Fakery in every rank}\label{s:fakes}

In this section we underscore the subtlety of recognising  standard apartments by describing a family of fake apartments in $\AF_n$ and $\OF_n$. This family shows that there exist fake apartments in $\OF_n$ with the property that each of their rank $(n-1)$ faces is standard and each of their rank 1 vertices is antipodal to the barycentre of the opposite face.

We consider a family of rank $n$
subgroups $H<F_n$ for which $\core_*(H)$ is obtained from the rose for
$\<a_1,\dots,a_{n-1}\>$ by connecting it to a loop labelled $a_n$ with a bridge labelled by a word of a particular form.
The words that we want are defined recursively: 
$$
W_0 := a_n  \text{ and }\  W_{k+1}:=W_k a_{k+1} W_k^{-1}.
$$
For example, 
$ W_2 = (a_na_1a_n^{-1})\, a_2 (a_na_1^{-1}a_n^{-1})$. Define
$$ H := \<a_1,\dots, a_{n-1}, W_{n-1}a_nW_{n-1}^{-1}\>.$$

\begin{lemma}\label{l:fakery}
For  
$j\le n-1$, the subgroup $V_j<F_n$ generated by $W_{n-1}a_nW_{n-1}^{-1}$ 
and $A_j=\{a_i \mid i\le n-1, \, i\neq j\}$ 
is a rank $(n-1)$ free factor antipodal to $\<W_{j-1} a_j W_{j-1}^{-1}\>$.
\end{lemma}

\begin{proof} We shall refer to the arc of $\core_*(V_k)$ joining the basepoint $*$ to the loop labeled $a_n$
as the {\em bridge}; it is labeled $W_{n-1}$. Let $p$ be the vertex on the bridge that is the terminus of the
path from $*$ labeled by the prefix $W_{j-1}\prec W_{n-1}$. We attach the lollipop $\core_*(W_{j-1} a_j W_{j-1}^{-1})$
to $\core_*(V_k)$ at $*$ and start folding. The  stalk of the lollipop folds entirely into 
$\core_*(V_k)$, at which point we have the graph obtained from $\core_*(V_j)$  by attaching
a loop labeled $a_j$ at $p$. 
The edge $e$ immediately beyond $p$ then folds around this loop and the folding continues as the arc
beyond $e$ that is
labeled $W_{j-1}^{-1}$ folds into the section of the bridge joining $p$ to $*$ -- at this point the folded graph is the wedge of the rose for $A_j$ and
two   lollipops, one with  stalk labeled $W_{j-1}$ and petal $a_j$, 
 and one with stalk $a_{j+1} W_j^{-1} a_{j+2} W_j\dots$ and petal
$a_n$. The initial edge on the stalk of the second lollipop folds into the rose, then the arc labeled $W_j^{-1}$ traces around the first lollipop, then the edge labeled $a_{j+2}$ folds into the rose, {\em etc.}

The folding continues until the entire stalk of the second
lollipop has folded into the wedge of the rose and the first lollipop. At this stage, the loop labeled $a_n$ shares its
vertex with the rose for $A_j$, and the stalk of the first lollipop folds into the rank-$(n-1)$ rose that they form.
Thus we obtain the rose $R_n$.
\end{proof}

\begin{proposition} The proper subsets of $\{a_1,\dots,a_{n-1},\, W_{n-1}a_nW_{n-1}^{-1}\}$ generate
free factors of $F_n$, and the subcomplex
 $\Delta < \AF_n$ that they
 span is an apartment with the following properties:
\begin{enumerate} 
\item every codimension-1 face is standard; 
\item the apartment is fake.
\end{enumerate}
The image of $\Delta$ in $\OF_n$ is also fake, and
\begin{enumerate} 
\item[(3)] the barycentre of each codimension-1 face is antipodal to the rank 1 factor opposite it. 
\end{enumerate} 
\end{proposition}

\begin{proof} Lemma \ref{l:fakery} assures us that each subset of cardinality $k<n$ generates a free factor of rank $k$, so 
$\Delta$ is indeed an apartment and (1) holds. 
The lemma also tells us that, in $\OF_n$, the codimension-1 face opposite the vertex $[a_j]$ is antipodal to $[a_j]=[W_{j-1} a_j W_{j-1}^{-1}]$, so (3) holds.

In a standard apartment of $\OF_n$, if $[A], [B]$ are the barycentres of distinct codimension-1 faces and $A,B$ are
representatives  with $A\cap B\neq 1$, then $A\cup B$ will generate $F_n$. But in the image of
$\Delta$, such representatives will only generate
$H \neq F_n$, so the apartment is fake.
\end{proof}

\bibliographystyle{alpha}
\bibliography{refBB} 

\end{document}